\documentclass{amsart}
\title{Stable Centres I: Wreath Products}
\author{Christopher Ryba}
\address{Department of Mathematics, University of California, Berkeley, CA 94720, USA}
\email{ryba@math.berkeley.edu}
\usepackage{hyperref}
\usepackage{fullpage}
\usepackage{amsmath}
\usepackage{amsfonts}
\usepackage{amssymb}
\usepackage{amsthm}
\usepackage{mathrsfs}  
\usepackage{ytableau}
\usepackage{subfig}
\usepackage{graphicx, caption}
\usepackage{comment}
\usepackage{bm}
\usepackage{todonotes}
\usepackage{float}

\usepackage{tikz-cd}

\DeclareMathOperator{\End}{End}

\newcommand{\bs}{\boldsymbol}
\newcommand{\RGamma}{\mathcal{R}_{\Gamma}}
\newcommand{\LambdaG}{\Lambda(\Gamma_*)}
\newcommand{\FH}{\mathrm{FH}}
\newcommand{\FHG}{\mathrm{FH}_{\Gamma}}

\newtheorem{theorem}{Theorem}[section]
\newtheorem{lemma}[theorem]{Lemma}
\newtheorem{proposition}[theorem]{Proposition}
\newtheorem{corollary}[theorem]{Corollary}
\newtheorem{definition}[theorem]{Definition}
\newtheorem{remark}[theorem]{Remark}
\newtheorem{example}[theorem]{Example}

\begin{document}
\maketitle
\begin{abstract}
A result of Farahat and Higman shows that there is a ``universal'' algebra, $\FH$, interpolating the centres of symmetric group algebras, $Z(\mathbb{Z}S_n)$. We explain that this algebra is isomorphic to $\mathcal{R} \otimes \Lambda$, where $\mathcal{R}$ is the ring of integer-valued polynomials and $\Lambda$ is the  ring of symmetric functions. Moreover, the isomorphism is via ``evaluation at Jucys-Murphy elements'', which leads to character formulae for symmetric groups. Then, we generalise this result to wreath products $\Gamma \wr S_n$ of a fixed finite group $\Gamma$. This involves constructing wreath-product versions $\RGamma$ and $\LambdaG$ of $\mathcal{R}$ and $\Lambda$, respectively, which are interesting in their own right (for example, both are Hopf algebras). We show that the universal algebra for wreath products, $\FHG$, is isomorphic to $\RGamma \otimes \LambdaG$ and use this to compute the $p$-blocks of wreath products.
\end{abstract}

\section{Introduction}
\noindent
Let $S_n$ be the $n$-th symmetric group, and $Z(\mathbb{Z}S_n)$ the centre of its group ring over the integers. The centre has a basis consisting of conjugacy class sums. Farahat and Higman \cite{FarahatHigman} showed that the structure constants of the multiplication with respect to this basis are integer-valued polynomials in $n$. For example, multiplying two transposition can give either an element with two cycles of size 2, a 3-cycle, or the identity. If $X_\mu^\prime$ is the sum of all elements of cycle type $\mu$ in $S_n$, so that $X_{(2,1^{n-2})}$ is the sum of all transpositions, we have
\begin{equation}
(X_{(2,1^{n-2})}^\prime)^2 = 2X_{(2^2,1^{n-4})}^\prime + 3X_{(3,1^{n-3})}^\prime + {n \choose 2}X_{(1^{n})}^\prime.
\end{equation}
Here, we emphasise that ${n \choose 2}$ is an integer-valued polynomial. This property of structure constants allowed Farahat and Higman to define an ``interpolating'' algebra, which we denote $\FH$, with coefficients in the ring of integer-valued polynomials, $\mathcal{R}$. See Section \ref{FH_theory_section} for a precise explanation. By construction, there are surjective ``specialisation'' homomorphisms $\FH \to Z(\mathbb{Z}S_n)$ for any $n$. The original motivation for defining the algebra $\FH$ was to give a new proof of Nakayama's Conjecture about $p$-blocks of symmetric groups.
\newline \newline \noindent
We begin by reviewing the construction of the algebra $\FH$, in particular we explain that it is isomorphic to $\mathcal{R} \otimes \Lambda$, where $\Lambda$ is the ring of symmetric functions. The Jucys-Murphy elements are key in describing this isomorphism, which leads to character formulae for $S_n$. We then fix a finite group $\Gamma$, and generalise the theory to the wreath products $\Gamma \wr S_n$. An algebra analogous to $\FH$ was defined by Wang in \cite{Wang1}, and we denote it $\FHG$. (Wang mostly worked with an associated graded version of $\FHG$ which he used to study Hilbert schemes of points on crepant resolutions of certain plane singularities.) The main result of this paper (Theorem \ref{main_thm}) determines the algebra structure of $\FHG$ via a ``Jucys-Murphy evaluation'' map that implies character formulae of the same nature as those for the symmetric group.
\newline \newline \noindent
We define two rings, $\RGamma$ and $\LambdaG$, which are, respectively, versions of $\mathcal{R}$ and $\Lambda$ weighted by the conjugacy classes of $\Gamma$. Then we show that $\FHG = \RGamma \otimes \LambdaG$, and mimic the approach of Farahat and Higman to classify the $p$-blocks of $\Gamma \wr S_n$.  A sequel paper will prove analogous results for Iwahori-Hecke algebras of type $A$, and consider connections to representations of $GL_n(\mathbb{F}_q)$.
\newline \newline \noindent
In the appendix, we discuss some properties of $\RGamma$ and $\LambdaG$ that are not needed for the main theory, but may be of independent interest. In particular, both $\RGamma$ and $\LambdaG$ are Hopf algebras. The Hopf algebra structure on $\RGamma$ is a consequence of the fact that $\RGamma$ is a certain distribution algebra. We describe homomorphisms from $\RGamma$ to a field; similarly to the case of $\mathcal{R}$, such maps are parametrised by several $p$-adic numbers.
\newline \newline \noindent
The paper is organised as follows. In Section 2, we review basic facts about symmetric groups, wreath products, and symmetric functions. We also briefly summarise the modular representation theory that is necessary for our applications.
In Section 3, we review the theory of the Farahat-Higman algebra $\FH$, Jucys-Murphy elements, and explain how $\FH$ is used to prove Nakayama's Conjecture.
In Section 4, we define $\LambdaG$, which we call the ring of $\Gamma_*$-weighted symmetric functions.
We discuss the wreath-product Farahat-Higman algebra $\FHG$ and introduce $\RGamma$ in Section 5.
Then we prove the isomorphism $\FHG = \RGamma \otimes \LambdaG$ in Section 6, and use it to study $p$-blocks of wreath products in Section 7. In the appendix we study maps from $\RGamma$ to a field, and show that both $\RGamma$ and $\LambdaG$ are Hopf algebras.

\section{Background}
\subsection{Symmetric Groups and Wreath Products}
We now introduce the notation and basic properties of symmetric groups and  wreath products. Details may be found in Chapter 1 of \cite{Macdonald} (wreath products are discussed in Appendix B).
\newline \newline \noindent
A \emph{partition}, $\lambda = (\lambda_1, \ldots, \lambda_r)$, is a finite non-increasing sequence of positive integers. The entries $\lambda_i$ are called the \emph{parts} of $\lambda$. It is also common to write $\lambda = (1^{m_1} 2^{m_2} \cdots)$ to mean that $\lambda$ has $m_i$ parts of size $i$ (this information uniquely determines $\lambda$). We will write $m_i(\lambda)$ for the number of parts of $\lambda$ of size $i$. The size of a partition is the sum of its parts: $|\lambda| = \lambda_1 + \cdots + \lambda_r$. If $|\lambda| = n$ it is common to say $\lambda$ is a partition of $n$, and to write $\lambda \vdash n$. The length of a partition is the number of parts: $l(\lambda) = r$. We write $\mathcal{P}$ for the set of partitions of any size.
\newline \newline \noindent
We depict partitions with \emph{Young diagrams}. The Young diagram of $\lambda$ consists of $l(\lambda)$ rows, where the $i$-th row consists of $\lambda_i$ boxes, left-justified. For example, the Young diagram of the partition $(5,2,1)$ is:
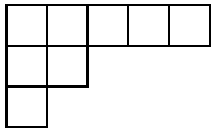
\begin{figure}[H]
\ydiagram{5,2,1}
\end{figure}
\noindent
The \emph{content} of the box in the $i$-th row from the top and $j$-th column from the left is defined to be $j-i \in \mathbb{Z}$. If the boxes of a Young diagram are labelled (usually with positive integers), we call it a \emph{Young tableau}. If $\lambda$ is a partition of $n$, and a Young diagram is labelled using the numbers $1,\ldots,n$ (each number used once) such that numbers increase in each row from top to bottom, and increase in each column from left to right, then the corresponding Young tableau is called a \emph{standard Young tableau}. Suppose that $\lambda \vdash n$. For a standard Young tableau $T$ of shape $\lambda$, an element $r \in \{1,\ldots,n\}$ labels a unique box of $T$; we write $c_T(r)$ for the content of the box in $T$ labelled $r$. For example, if $T$ is the standard Young tableau of shape $(3,3,1)$ below,
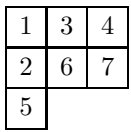
\begin{figure}[H]
\ytableausetup{centertableaux}\begin{ytableau}
1 & 3 & 4 \\
2 & 6 & 7 \\
5
\end{ytableau}
\end{figure}
\noindent
then $c_T(5) = 1-3 = -2$, $c_T(6) = 2-2 = 0$, and $c_T(7) = 3-2 = 1$.
\newline \newline \noindent
A \emph{border strip} $R$ of $\lambda$ is a subset of the boxes of the Young diagram of $\lambda$ satisfying the following conditions. Firstly, $\lambda \backslash R$ (i.e. the diagram of $\lambda$ with the border strip $R$ removed) should be the Young diagram of a partition. Secondly, the subset $R$ should be contiguous (boxes sharing an edge are considered adjacent, but those sharing only a vertex are not). Finally, $R$ should not contain any $2\times2$ square of boxes. Below are some examples of border strips $R$ for the partition $(4,2,2)$, where in each case $R$ is indicated by the shaded squares:
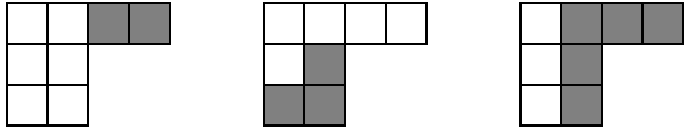
\begin{figure}[H]
\centering
\ytableausetup{nosmalltableaux}
\begin{ytableau}
*(white) &*(white) &*(gray) &*(gray) \\
*(white) &*(white) \\
*(white) &*(white)
\end{ytableau}
\hspace{10mm}
\begin{ytableau}
*(white) &*(white) &*(white) &*(white) \\
*(white) &*(gray) \\
*(gray) &*(gray)
\end{ytableau}
\hspace{10mm}
\begin{ytableau}
*(white) &*(gray) &*(gray) &*(gray) \\
*(white) &*(gray) \\
*(white) &*(gray)
\end{ytableau}
\end{figure} 
\noindent
If $p$ is a prime number, the \emph{$p$-core} of $\lambda$ is the partition obtained by successively removing border strips of size $p$ from the diagram of $\lambda$ until it is no longer possible to do so. It turns out that the resulting partition is independent of the choice of how to remove rim hooks of size $p$.
\begin{example} \label{2-core_example}
The $2$-cores of the partitions $(3)$, $(2,1)$, and $(1,1,1)$ are $(1)$, $(2,1)$, and $(1)$ respectively. The diagrams below illustrate this (grey boxes indicate a border strip of size 2 to be removed):
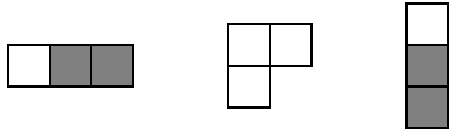
\begin{figure}[H]
\centering
\ytableausetup{nosmalltableaux}
\begin{ytableau}
*(white) &*(gray) &*(gray)
\end{ytableau}
\hspace{10mm}
\begin{ytableau}
*(white) &*(white) \\
*(white)
\end{ytableau}
\hspace{10mm}
\begin{ytableau}
*(white) \\
*(gray) \\
*(gray)
\end{ytableau}
\end{figure}
\end{example}
\noindent
Consider a border strip $R$ of size $p$, viewed as a sequence of adjacent boxes, starting at the bottom-left-most square, with each subsequent box either above or to the left of the previous one. Then the content of each subsequent box is $1$ larger than the content of the preceding box as either the column index $j$ increases or the row index $i$ decreases. This means that the content of the boxes in $R$ attain each congruence class in $\mathbb{Z}/p\mathbb{Z}$ exactly once. It follows that two partitions of the same size with the same $p$-core have the same multiset of content modulo $p$. In fact, two partitions of the same size have the same multiset of content modulo $p$ if and only if the two partitions have the same $p$-core (e.g. in Example \ref{2-core_example} above). Details can be found in Examples 8 and 11 in Section 1.1 of \cite{Macdonald}.
\newline \newline \noindent
The \emph{$n$-th symmetric group}, $S_n$, is the set of bijections from the set $\{1, \ldots, n\}$ to itself, which is a group with the operation of function composition. The \emph{cycle type} of an element $\sigma \in S_n$ is a partition $\lambda$ such that $m_i(\lambda)$ is the number of cycles of size $i$ in $\sigma$. Two elements of $S_n$ are conjugate if and only if they have the same cycle type, so the conjugacy classes of $S_n$ are in bijection with partitions of size $n$.
\newline \newline \noindent
Let $\Gamma$ be a finite group. We let $\Gamma_* = \{c_1, \ldots, c_l\}$ be the set of conjugacy classes of $\Gamma$. Sometimes we will write $1$ to denote the identity conjugacy class. We will need to consider the set $\Gamma^*$ of irreducible representations of $\Gamma$ over $\mathbb{C}$ (we could work with any characteristic-zero splitting field of $\Gamma$, but we stick to the complex numbers for concreteness). By abuse of notation, we write $\chi \in \Gamma^*$ to mean both the irreducible representation and its character, for example we may write $\chi(1) = \dim(\chi)$. The centre of the group algebra $\mathbb{Z}\Gamma$ is a free $\mathbb{Z}$-module with a basis consisting of conjugacy-class sums. By further abuse of notation, we use the same symbol to denote a conjugacy class and the sum of its elements. This allows us to write $Z(\mathbb{Z}\Gamma) = \mathbb{Z}\Gamma_*$. In addition, we define the structure constants $A_{i,j}^k \in \mathbb{Z}$ via the following equations in $\mathbb{Z}\Gamma_*$:
\[
c_i c_j = \sum_k A_{i,j}^k c_k.
\]
By the Artin-Wedderburn theorem, $\mathbb{C}\Gamma$ is a product of matrix algebras indexed by $\Gamma^*$. Since the centre of a matrix algebra is one dimensional, we have an isomorphism of algebras $\mathbb{C}\Gamma_* = \mathbb{C}^{\Gamma^*}$, where the implied basis of the latter space consists of orthogonal central idempotents in $\mathbb{C}\Gamma_*$.
\newline \newline \noindent
The $n$-fold product of $\Gamma$ with itself has an action of the symmetric group by permutation of factors. Concretely, if $(g_1, \ldots, g_n) \in \Gamma^n$ and $\sigma \in S_n$, then
\begin{equation} \label{semidirect_mult}
\sigma (g_1, \ldots, g_n) = (g_{\sigma^{-1}(1)}, \ldots, g_{\sigma^{-1}(n)}).
\end{equation}
This defines an automorphism of $\Gamma^n$, and in fact we obtain a homomorphism $S_n \to \mathrm{Aut}(\Gamma^n)$. The corresponding semidirect product $\Gamma^n \rtimes S_n$ is called the \emph{wreath product} of $\Gamma$ with $S_n$, and is denoted $\Gamma \wr S_n$. As a set, $\Gamma \wr S_n$ is equal to $\Gamma^n \times S_n$. We write an element $(g_1, \ldots, g_n, \sigma)$ as $(\mathbf{g}, \sigma)$ where $\mathbf{g} \in \Gamma^n$. Then the group operation is
\[
(\mathbf{g}, \sigma) (\mathbf{h}, \rho) = (\mathbf{g} \sigma(\mathbf{h}), \sigma \rho),
\]
where $\sigma(\mathbf{h})$ has the meaning in Equation \ref{semidirect_mult}. Both $\Gamma^n$ and $S_n$ are subgroups of $\Gamma \wr S_n$ in the obvious way. If $g \in \Gamma$, we write $g^{(i)} \in \Gamma^n$ for the element whose $i$-th component is $g$, and all other components are the identity. This notation extends linearly to give us an embedding of $\mathbb{Z}\Gamma$ into $\mathbb{Z}\Gamma^n$, so for example we may write $c^{(i)} = \sum_{g \in c} g^{(i)}$ for the conjugacy-class sum $c$ embedded in the $i$-th component of $\mathbb{Z}\Gamma^n$. We also let $\mathbf{g}_i$ be the $i$-th entry in $\mathbf{g} \in \Gamma^n$, so we have the tautological equality
\[
\mathbf{g} = \prod_{i=1}^n \mathbf{g}_i^{(i)}.
\]
This also makes it simpler to write the $S_n$ action:
\[
\sigma(\mathbf{g}) = \prod_{i=1}^n \mathbf{g}_{i}^{(\sigma(i))}.
\]
\begin{example}
An important special case is $\Gamma = C_2$, the cyclic group of order 2. In that case, $C_2 \wr S_n$ is called \emph{the $n$-th hyperoctahedral group}. It arises in Lie theory as the Weyl group of types $B_n$ and $C_n$. 
We will use the case $\Gamma = C_2$ for examples throughout the paper.
\end{example}
\noindent
We now describe the conjugacy classes in $\Gamma \wr S_n$. Suppose that $(\mathbf{g}, \sigma) \in \Gamma \wr S_n$, and that $(i_1, \ldots, i_r)$ is a cycle of $\sigma \in S_n$. We also say that it is a cycle of $(\mathbf{g}, \sigma)$, and we define its \emph{type} to be the conjugacy class of $g_{i_r} \cdots g_{i_1}$. Note that $g_{i_1}(g_{i_r} \cdots g_{i_1})g_{i_1}^{-1} = g_{i_1} g_{i_r}\cdots g_{i_2}$, so that $g_{i_r} \cdots g_{i_1}$ and $g_{i_1} g_{i_r} \cdots g_{i_2}$ are conjugate. This means that the two (equal) cycles $(i_1, \ldots, i_r)$ and $(i_2, \ldots, i_r, i_1)$ have the same type, and iterating though all cyclic permutations we see that the type of a cycle is well defined. We record the sizes and types of an element of $\Gamma \wr S_n$ in a multipartition.
\begin{definition}
If $D$ is a set, we say that a \emph{multipartition} indexed by $D$ is a function from $D$ to the set of partitions. We denote the set of multipartitions indexed by $D$ by $\mathcal{P}(D)$. If $D$ is not specified, we take $D = \Gamma_*$. The \emph{size} of a multipartition $\bs \lambda$ is the sum of the sizes of its constituent partitions:
\[
|\bs \lambda| = \sum_{x \in D} |\bs\lambda(x)|.
\]
Similarly, the \emph{length} of a multipartition is the sum of the lengths of its consituent partitions:
\[
l(\bs \lambda) = \sum_{x \in D} l(\bs\lambda(x)).
\]
To express a multipartition we write the juxtaposition of its constituent partitions with the corresponding element of $D$ as a subscript (omitting empty partitions for brevity). If $\bs\lambda(y)$ is the empty partition for $y \neq x$, we say that $\bs\lambda$ is \emph{concentrated in type $x$}.
\end{definition}
\noindent
For example, $(2,1)_{c_1} (1,1,1)_{c_2}$ is the multipartition in $\mathcal{P}(\Gamma_*)$ taking the value $(2,1)$ at $c_1 \in \Gamma_*$ and taking the value $(1,1,1)$ at $c_2 \in \Gamma_*$, while $(3,1)_{c_3}$ is concentrated in type $c_3$.
\begin{definition}
The \emph{cycle type} of an element $(\mathbf{g}, \sigma)$ of $\Gamma \wr S_n$ is the $\Gamma_*$-indexed multipartition $\bs \mu$ such that for each $i \in \mathbb{Z}_{>0}$ and $c \in \Gamma_*$, $m_i(\bs\mu(c))$ is equal to the number of cycles of type $c$ and size $i$ in $(\mathbf{g}, \sigma)$.
\end{definition}
\noindent
By construction, the size of the cycle type of an element of $\Gamma \wr S_n$ is $n$. Two elements of $\Gamma \wr S_n$ are conjugate if and only if they have the same cycle type. So the conjugacy classes of $\Gamma \wr S_n$ correspond to $\Gamma_*$-indexed multipartitions of $n$.

\begin{lemma} \label{cycle_label_compatibility}
Let $\sigma = (i_1, \ldots, i_r)$ be an $r$-cycle in $S_n$. Let $X_\sigma(c) \in \mathbb{Z}\Gamma \wr S_n$ be the sum of all elements of the form $(\mathbf{g}, \sigma) \in \Gamma \wr S_n$ where the type of the cycle $\sigma$ is $c \in \Gamma_*$, and $\mathbf{g}_i = 1$ for all $i$ not in the cycle. Then for any $i_j$ in the cycle, we have
\[
X_\sigma(c) = c^{(i_j)}X_\sigma(1) = X_\sigma(1) c^{(i_j)},
\]
where $c^{(i)} = \sum_{g \in c} g^{(i)}$ is the conjugacy class sum $c$ embedded in the $i$-th component of $\mathbb{Z}\Gamma^n$. 
\end{lemma}
\begin{proof}
The second equality follows from the first because
\[
(\mathbf{g}, \sigma) c^{(i_j)}  = (\mathbf{g}c^{(i_{j+1})}, \sigma) = c^{(i_{j+1})}(\mathbf{g}, \sigma),
\]
where $i_{r+1}$ is taken to mean $i_1$ when $j=r$. In the second step we use the fact that $c$ is a central element of $\mathbb{Z}\Gamma^n$. We write
\[
X_\sigma(c) = \sum_{\mathbf{g}_{i_r} \cdots \mathbf{g}_{i_1} \in c} (\mathbf{g}, \sigma)
\]
where $\mathbf{g}_i = 1$ for $i \neq i_1, \ldots, i_r$. The condition $\mathbf{g}_{i_r} \cdots \mathbf{g}_{i_1} \in c$ may be written as
\[
\mathbf{g}_{i_j} \in \mathbf{g}_{i_{j+1}}^{-1} \cdots \mathbf{g}_{i_{r}}^{-1}c\mathbf{g}_{i_1}^{-1} \cdots \mathbf{g}_{i_{j-1}}^{-1} = c (\mathbf{g}_{i_{j+1}}^{-1} \cdots \mathbf{g}_{i_{r}}^{-1}\mathbf{g}_{i_1}^{-1} \cdots \mathbf{g}_{i_{j-1}}^{-1}),
\]
where we have used the fact that $c$ is conjugation invariant to move it to the front of the product. In the case where $c=1$, choosing all the elements other than $\mathbf{g}_{i_j}$ arbitrarily uniquely determines $\mathbf{g}_j$. In the case of general $c$, we again may choose the elements other than $\mathbf{g}_{i_j}$ arbitrarily, and then $\mathbf{g}_{i_j}$ may be any element of $c$ multiplied by the value of $\mathbf{g}_{i_j}$ from the $c=1$ case.
\end{proof}
\begin{lemma} \label{cycle_label_product}
Let $r<s$ be positive integers and let $\sigma = (i_1, \ldots, i_r)$ and $\rho = (i_r, \ldots, i_s)$ be cycles in $S_n$ of lengths $r$ and $s-r+1$ respectively that intersect at a single element $i_r$. Let $X_\sigma(c_i)$ and $X_\rho(c_j)$ be as in Lemma \ref{cycle_label_compatibility}. Then
\[
X_\sigma(c_i) X_\rho(c_j) = \sum_{k} A_{i,j}^k X_{\sigma\rho}(c_k).
\]
\end{lemma}
\begin{proof}
We may assume $c_i = c_j = 1$ because
\[
X_\sigma(c_i) X_\rho(c_j) = c_i^{(i_r)} X_{\sigma}(1) c_j^{(i_r)} X_{\rho} (1) = c_i^{(i_r)}c_j^{(i_r)} X_\sigma(1) X_\rho(1),
\]
and
\[
\sum_{k} A_{i,j}^k X_{\sigma\rho}(c_k) = \sum_{k} A_{i,j}^k c_k^{(i_r)} X_{\sigma\rho}(1) = c_i^{(i_r)}c_j^{(i_r)} X_{\sigma\rho}(1).
\]
Now we suppose that $(\mathbf{g}, \sigma)$ has all cycles of type $1$, so in particular
\[
\mathbf{g}_{i_r} \cdots \mathbf{g}_{i_1} = 1.
\]
Similarly we consider $(\mathbf{h}, \rho)$ with
\[
\mathbf{h}_{i_s} \cdots \mathbf{h}_{i_r} = 1.
\]
Then $(\mathbf{g}, \sigma)(\mathbf{h}, \rho) = (\mathbf{g} \sigma(\mathbf{h}), \sigma\rho)$, and we have 
\[
(\mathbf{g} \sigma(\mathbf{h}))_{i_p} =
 \left\{
        \begin{array}{ll}
            \mathbf{g}_{i_1} \mathbf{h}_{i_r} & \quad p=1 \\
            \mathbf{g}_{i_p} & \quad 1 < p \leq r \\
            \mathbf{h}_{i_p} & \quad r < p \leq s
        \end{array}
    \right.
\]
The type of this cycle is the conjugacy class of
\[
\mathbf{h}_{i_s} \cdots \mathbf{h}_{i_{r+1}} \mathbf{g}_{i_r} \cdots \mathbf{g}_{i_1}\mathbf{h}_{i_r}.
\]
But now, $\mathbf{g}_{i_r} \cdots \mathbf{g}_{i_1} = 1$, leaving $\mathbf{h}_{i_s} \cdots \mathbf{h}_{i_r}$ which equals $1$. Moreover the elements $\mathbf{g}_{i_2}, \ldots, \mathbf{g}_{i_r}$ and $\mathbf{h}_{i_{r+1}}, \ldots, \mathbf{h}_{i_s}$ may be chosen arbitrarily, and then $\mathbf{g}_{i_1}$ and $\mathbf{h}_{i_r}$ are determined by the condition that the type of the cycle should be $1$. As a result we get every term in the sum $X_{\sigma\rho}(1)$ exactly once.
\end{proof}

\noindent
We now turn our attention to describing the irreducible representations of $\Gamma \wr S_n$ over $\mathbb{C}$ in terms of those of $\Gamma$ and $S_n$. Recall that the irreducible representations $S^\lambda$ of the symmetric group $S_n$ in characteristic zero are called Specht modules and are labelled by partitions $\lambda$ of size $n$. Analogously, the irreducible representations $V^{\bs\lambda}$ of $\Gamma \wr S_n$ over $\mathbb{C}$ are in bijection with by $\Gamma^*$-indexed multipartitions $\bs\lambda$ of size $n$. We explain how they can be constructed.
\newline \newline \noindent
If $\tau = (\tau_1, \ldots, \tau_l)$ is a sequence of non-negative integers adding to $n$, then the product of the symmetric groups $S_{\tau_i}$ is a subgroup of $S_n$, called a \emph{Young subgroup}, and denoted $S_\tau$. Here the factors $S_{\tau_i}$ are viewed as a subgroups of $S_n$ by permuting disjoint contiguous blocks of $\tau_i$ elements of $\{1,\ldots, n\}$. We refer to the $S_{\tau_i}$ as the \emph{factor groups} of the Young subgroup $S_\tau$. The product
\[
\prod_{i} (\Gamma \wr S_{\tau_i}) = \Gamma^n \rtimes S_\tau
\]
is a subgroup of $\Gamma \wr S_n$ which we denote $\Gamma \wr S_\tau$. Secondly, if $\chi \in \Gamma^*$ and $\lambda$ is a partition of $\tau_i$, then 
\[
\chi^{\otimes \tau_i} \otimes S^{\lambda}
\]
is a representation of $\Gamma \wr S_{\tau_i}$ where $\Gamma^{\tau_i}$ acts on the first factor in the natural way, while $S_{\tau_i}$ acts by permuting the tensor factors in the first term, and in the usual way on the second term. Finally, the irreducible representation of $\Gamma \wr S_n$ indexed by $\bs\lambda$ is the following induced representation:
\[
V^{\bs\lambda} =
\mathrm{Ind}_{\Gamma \wr S_{\tau}}^{\Gamma \wr S_n} \left( \bigotimes_{\chi \in \Gamma^*} \chi^{\otimes \tau_\chi} \otimes S^{\bs\lambda(\chi)}
\right),
\]
where $\tau$ is the composition of $n$ whose parts are $\tau_\chi = |\bs\lambda(\chi)|$, indexed by $\chi \in \Gamma^*$ in some order.

\subsection{Symmetric Functions} 
\noindent
We recall the construction of the ring of symmetric functions, $\Lambda$. Macdonald's book \cite{Macdonald} is the authoritative reference on this topic. In Section \ref{LambdaG_section}, we will mimic this construction to define a related ring, $\LambdaG$.
\newline \newline \noindent
The $n$-th tensor power of $\mathbb{Z}[x]$ may be viewed as the polynomial ring $\mathbb{Z}[x_1, \ldots, x_n]$. The action of the symmetric group $S_n$ by permutation of tensor factors may be understood as permutation of the variables. The $S_n$-action preserves the grading where each variable has degree $1$, so the ring of $S_n$-invariants inherits the same grading. An important family of invariant elements are the \emph{monomial symmetric polynomials}, $m_\lambda$, indexed by partitions $\lambda$. By definition, $m_\lambda(x_1, \ldots, x_n)$ is the sum of all monomials $x_1^{a_1} x_2^{a_2} \cdots x_n^{a_n}$ such that the number of exponents equal to $i \in \mathbb{Z}_{>0}$ is $m_i(\lambda)$.
\newline \newline \noindent
Let $\Lambda_n$ be the $S_n$-invariants of $\mathbb{Z}[x_1, \ldots, x_n]$, and $\Lambda_n^{k}$ be the degree $k$ component of $\Lambda_n$. For $m > n$ we have a map $\rho_{m,n}: \Lambda_m^k \to \Lambda_n^k$ defined by setting the variables $x_{n+1}, \ldots, x_m$ to zero. For $m > k > n$, we have $\rho_{m,k} \circ \rho_{k,n} = \rho_{m,n}$, and so these $\rho_{m,n}$ define an inverse system on the $\mathbb{Z}$-modules $\Lambda_n^k$ (for fixed $k$ and $n \in \mathbb{Z}_{\geq 0}$), and we take
\[
\Lambda^k = \varprojlim \Lambda_n^{k}.
\]
In explicit terms, $\Lambda^k$ consists of sequences $(q_1, q_2, \ldots)$, where $q_n \in \Lambda_n^k$, that satisfy the condition $\rho_{m,n}(q_m) = q_n$. It is well known that $\rho_{m,n}$ is an isomorphism when $m > n\geq k$, and in that case $\Lambda_n^k$ has a basis of monomial symmetric functions $m_\lambda$ indexed by partitions $\lambda$ of size $k$. Moreover, the monomial symmetric functions with different numbers of variables are stable with respect to the maps $\rho_{m,n}$, i.e.
\[
m_\lambda(x_1, \ldots, x_n) = \rho_{m,n}(m_\lambda(x_1, \ldots, x_m)) = m_{\lambda}(x_1, \ldots, x_n, 0, \ldots, 0).
\]
Therefore the monomial symmetric polynomials define elements of the inverse limit $\Lambda^k$, which are called \emph{monomial symmetric functions} and denoted $m_\lambda$. From this it follows that $\Lambda^k$ is a free $\mathbb{Z}$-module with basis $m_\lambda$ indexed by partitions $\lambda$ of size $k$. Finally, the ring of symmetric functions is
\[
\Lambda = \bigoplus_{k=0}^\infty \Lambda^k,
\]
which is free as a $\mathbb{Z}$-module with basis consisting of the monomial symmetric functions. In fact, $\Lambda$ is a graded ring because the operation of setting a variable $x_m$ to zero respects degree and multiplication. By general properties of inverse limits, we have a canonical ring homomorphism $\Lambda \to \Lambda_n$ by sending $(q_1, q_2, \ldots)$ to $q_n$.
\newline \newline \noindent
The \emph{elementary symmetric functions} $e_r$ are elements of $\Lambda$ given by $e_r = m_{(1^r)}$. When evaluated in $n$ variables they are
\[
e_r(x_1, \ldots, x_n) = \sum_{i_1 < i_2 < \cdots < i_r} x_{i_1} x_{i_2} \cdots x_{i_r},
\]
namely, the sum of all products of $r$ distinct variables. Often it is convenient to work with generating functions, in which case we have
\[
\sum_{r \geq 0} e_r(x_1, \ldots, x_n) t^r = \prod_{i=1}^n (1 + x_i t).
\]
A key fact is that the elementary symmetric functions generate $\Lambda$ as a free polynomial algebra: $\Lambda = \mathbb{Z}[e_1, e_2, \ldots ]$.

\subsection{Central Characters and Blocks} \label{mod_rep_thy_subsect}
\noindent
In this subsection, we review the modular representation theory needed for our applications and for the appendix.
\newline \newline \noindent
The notion of central character is a slight variation of a character.
\begin{definition}
Suppose that $\Gamma$ is a finite group. The \emph{central character} $\omega_c^\chi$ of a conjugacy class $c \in \Gamma_*$ on an irreducible representation $\chi \in \Gamma^*$ is the scalar by which $c \in \mathbb{Z}\Gamma$ acts on $\chi$ (central elements always act by scalar multiplication on irreducible representations).
\end{definition}
\noindent
There is a formula for the central characters in terms of the usual characters:
\[
\omega_c^\chi = \frac{\sum_{g \in c} \chi(g)}{\chi(1)}.
\]
This formula is obtained by taking the trace of $c$, viewed as a linear operator on the irreducible representation, and dividing by the dimension (which is the trace of the identity). Since characters are constant on conjugacy classes,
\[
\omega_c^\chi = \frac{|c|\chi(g)}{\chi(1)},
\]
where $|c|$ is the size of the conjugacy class $c$, and $g$ is any element of $c$.
\newline \newline \noindent
Central characters satisfy a variety of properties, for example they are always algebraic integers (see Proposition 5.3.2 of \cite{EtingofEtAl}). We will use the fact that the central characters control the blocks of the modular representations of $\Gamma$. We briefly review some of the key definitions and properties of blocks.
\newline \newline \noindent
For the rest of this section, $\mathbb{F}$ is an algebraically closed field of any characteristic. We still have $Z(\mathbb{F}\Gamma) = \mathbb{F}\Gamma_*$. Recall that the centre $\mathbb{F}\Gamma_*$ acts on a simple $\mathbb{F}\Gamma$-module $M$ by multiplication by scalars; multiplication by an element of the centre commutes with the module action, and hence defines an element of $\End(M)$, which equals $\mathbb{F}$ by Schur's Lemma. This gives a homomorphism $\mathbb{F}\Gamma_* \to \mathbb{F}$, which is also called a central character. When $\mathbb{F} = \mathbb{C}$ and $M = \chi \in \Gamma^*$, $\omega_c^\chi$ is the value of this homomorphism on the conjugacy-class sum $c$.

\begin{definition}
The \emph{blocks} of the group algebra $\mathbb{F}\Gamma$ are the minimal indecomposable two-sided ideals $B_i$ of $\mathbb{F}\Gamma$.
\end{definition}
\noindent
It is well known that $\mathbb{F}\Gamma$ is the direct sum of its constituent blocks:
\begin{equation} \label{block_sum}
\mathbb{F}\Gamma = \bigoplus_i B_i.
\end{equation}
In particular, the intersection of distinct blocks is zero. Because the blocks are ideals, $B_i B_j \subseteq B_i \cap B_j = 0$ for $i \neq j$. This means that Equation \ref{block_sum} is a decomposition of $\mathbb{F}$-algebras. So if we express the identity element of $\mathbb{F}\Gamma$ as $1 = \sum_i e_i$ with $e_i \in B_i$, it follows that that $e_i$ is the identity element of $B_i$, and moreover $B_i = e_i \mathbb{F}\Gamma e_i$. The upshot of this is that if $M$ is any $\mathbb{F}\Gamma$-module, then
\[
M = \bigoplus_i e_i M.
\]
So if $M$ is indecomposable, $e_i M$ is nonzero for exactly one value of $i$, and for that particular $i$, $e_i M = M$. This makes $M$ into a module for some $B_i$. It is common to say that $M$ belongs to the block $B_i$. As a result, we have a decomposition of module categories
\[
\mathbb{F}\Gamma - \mathrm{mod} = \bigoplus_i \left(B_i -\mathrm{mod}\right),
\]
because any module $M$ splits as a direct sum of modules $e_iM$ belonging to each block $B_i$.
We can determine when two simple modules belong to the same block in terms of central characters.

\begin{proposition} \label{central_char_blocks}
Two simple modules for $\mathbb{F}\Gamma$ belong to the same block if and only if every element $Z(\mathbb{F}\Gamma)$ acts on each of them by the same scalar (i.e. they have the same central character). 
\end{proposition}
\begin{proof}
Suppose that $B_i$ is a block of $\mathbb{F}\Gamma$. Lemma 4.1 of \cite{modrepfingrp} shows that $Z(B_i)/J(Z(B_i)) = \mathbb{F}$, where $J$ is the Jacobson radical. According to  \cite{modrepfingrp}, a central character of a block $B_i$ is defined as the composite homomorphism
\[
Z(\mathbb{F}\Gamma) \to Z(B_i) \to Z(B_i)/J(Z(B_i)) = \mathbb{F}.
\]
By definition, the Jacobson radical acts by zero on any simple module, so the action of $Z(\mathbb{F}\Gamma)$ on a simple module factors through the above homomorphism. This shows that this definition of central character is consistent with ours. Moreover, two simple modules in the same block have the same central character. Simple modules in different blocks must have different central characters, because the central idempotents $e_i$ act differently ($e_i$ acts by the identity or zero depending on whether the module belongs to $B_i$ or not).
\end{proof}
\noindent
There is a procedure for taking a representation of $\Gamma$ in characteristic zero, and producing a representation in positive characteristic. The details are technical, so we only sketch the main idea, directing the interested reader to \cite{modrepfingrp}. We begin with a \emph{$p$-modular system} (here $p$ is a prime), which is a triple $(K, \mathcal{O}, k)$ defined as follows. Firstly, $\mathcal{O}$ is a complete discrete valuation ring. Secondly, $k$ is the residue field of $\mathcal{O}$ which is required to be of characteristic $p$. Thirdly, $K$ is the fraction field of $\mathcal{O}$, which is required to be of characteristic zero. We recall several facts:
\begin{itemize}
\item We may find a $p$-modular system such that every simple representation of $\Gamma$ over either $K$ or $k$ is absolutely irreducible. For example, we may take this $K$ to be a finite extension of the $p$-adic numbers, $\mathbb{Q}_p$, in which the polynomial $x^{|\Gamma|}-1$ splits. Then we may also take $\mathcal{O}$ to be the integral closure of $\mathbb{Z}$ in $K$, which makes $k$ a finite field. (See Section 3.3 of \cite{modrepfingrp}.)
\item The representation theory of $\Gamma$ over $\mathbb{C}$ and $K$ is essentially the same. In either case, irreducible representations are defined over $\mathbb{Q}[\zeta]$, where $\zeta$ is a primitive $|\Gamma|$-th root of unity. Then an identification of $\mathbb{Q}[\zeta]$ as a subring of $\mathbb{C}$ with $\mathbb{Q}[\zeta]$ as a subring of $K$ gives a correspondence of irreducible representations that respects central characters. (This follows from Theorem 2.7A of \cite{modrepfingrp}.)
\item Given a $K\Gamma$-module $M$, we may find a $\Gamma$-stable free $\mathcal{O}$-submodule $M_{\mathcal{O}}$ such that $KM_{\mathcal{O}} = M$ (i.e. $M_{\mathcal{O}}$ spans $M$ over $K$). Taking the quotient by the maximal ideal of $\mathcal{O}$ gives a $k\Gamma$-module, $M_k =M_{\mathcal{O}} \otimes_{\mathcal{O}} k$. While the isomorphism class of $M_k$ depends on the choice of integral form $M_{\mathcal{O}}$, the composition factors (with multiplicity) of $M_k$ do not. (This is Theorem 3.6 of \cite{modrepfingrp}.)
\end{itemize}
As a result, the operation $M \to M_k$ is well-defined on the level of Grothendieck groups:
\[
d: K_0(K\Gamma-\mathrm{mod}) \to K_0(k\Gamma-\mathrm{mod}).
\]
The map $d$ is called the \emph{decomposition matrix}, and its entries (with respect to the bases coming from simple modules) are called \emph{decomposition numbers}. These definitions turn out to be independent of the choice of $p$-modular system. Decomposition numbers are poorly understood outside of a few special cases, and are an active area of research (one recent breakthrough was the disproof of the \emph{James conjecture} \cite{Williamson} about decomposition numbers for the symmetric groups).
\newline \newline \noindent
An easy observation from the form of the map $d$, is that since the centre $K\Gamma_*$ acts by scalars on a simple $K\Gamma$-module $M$, any conjugacy class sum $c$ acts by scalars on $M_{\mathcal{O}}$ and also $M_{k}$. This implies that every composition factor of $M_k$ has the same central character. So if $M$ is a simple $K\Gamma$-module, $M_{k}$ belongs to a single block $B_i$. The (necessarily disjoint) subsets of the simple (characteristic zero) modules mapping to  a given (characteristic $p$) block are called \emph{$p$-blocks}. The $p$-blocks determine a block-diagonal structure of the decomposition matrix. We will determine the $p$-blocks of $\Gamma \wr S_n$ in terms of the $p$-blocks of $\Gamma$ in Theorem \ref{wreath_nakayama_thm}. This will use the following tool.
\begin{proposition}[Theorem 4.2B, \cite{modrepfingrp}] \label{block_prop}
Two irreducible representations $\chi_1, \chi_2$ of $K\Gamma$ are in the same $p$-block if and only if
\[
\omega_c^{\chi_1} \equiv \omega_c^{\chi_2} \hspace{5mm} (\mbox{$\mathrm{mod}$ $\pi$}),
\]
for all $c \in \Gamma_*$, where $\pi$ is a uniformiser for $\mathcal{O}$.
\end{proposition}
\begin{proof}
The $p$-block of an irreducible representation $\chi$ is determined by the central character of any composition factor of $(\chi)_k$ by Proposition \ref{central_char_blocks}. But the central character can be computed by taking the central character of $\chi$ and passing to $k = \mathcal{O}/(\pi)$. Since $\Gamma_*$ is a basis for the group algebra of $\Gamma$, the equality of central characters is equivalent to the stated equations.
\end{proof}
\noindent
Our classification will use the following well-known result.
\begin{proposition} \label{brauer_nesbitt_prop}
Suppose that $|\Gamma| = p^r m$, where $p \nmid m$ and $M$ is a simple $K\Gamma$-module. Then the block $B_i$ to which $M_k$ belongs is semisimple (as a $k$-algebra) if and only if $p^r | \dim_K(M)$, in which case $M_k$ is simple.
\end{proposition}
\begin{proof}
The ``if'' direction is known as the Brauer-Nesbitt theorem (although there are also other results with that name). On the other hand, if $B_i$ is semisimple, then the $B_i$-module $M_k$ is projective for $B_i$ and therefore for $k\Gamma$. If $P$ is a $p$-Sylow subgroup of $\Gamma$, then $M_k$ is a projective $kP$-module. But since $p$-groups only have one irreducible representation in characteristic $p$ (the trivial representation), $kP$ is a (not necessarily commutative) basic local algebra. This implies that a projective $kP$-module is free. In particular, $\dim_k(M_k)$ is a multiple of $\dim_k(kP) = p^r$.
\end{proof}

\section{The Farahat-Higman algebra} \label{FH_theory_section}
\noindent
Let $X_{\mu}^\prime$ be the sum of all elements of cycle type $\mu$ in $S_n$. These conjugacy-class sums define a $\mathbb{Z}$-basis of $Z(\mathbb{Z}S_n)$ indexed by partitions of $n$. For example
\[
X_{(2,1^{n-2})}^\prime = \sum_{i<j} (i,j)
\]
which is the sum of all elements of cycle type $(2,1^{n-2})$ (i.e. transpositions) in $S_n$. Of course $Z(\mathbb{Z}S_n)$ is a ring, so the product of two conjugacy-class sums should be a linear combination of such sums. Taking the example from the introduction,
\begin{equation} \label{eqn:FH_alg_example}
(X_{(2,1^{n-2})}^\prime)^2 = 2X_{(2^2,1^{n-4})}^\prime + 3X_{(3,1^{n-3})}^\prime + {n \choose 2}X_{(1^{n})}^\prime,
\end{equation}
which we now verify by considering the product of two arbitrary transpositions $(i,j)$ and $(k,l)$. If $i,j,k,l$ are distinct, we have
\[
(i,j)(k,l) = (k,l)(i,j),
\]
so each element of cycle type $(2^2,1^{n-4})$ appears twice. If $\{i,j\}$ and $\{k,l\}$ have one element in common, $(i,j)(k,l)$ is a 3-cycle. Since
\[
(i,j,k) = (i,j)(j,k) = (k,i)(i,j) = (j,k)(k,i),
\]
$X_{(3,1^{n-3})}^\prime$ appears with multiplicity 3. Finally, if $(i,j)$ and $(k,l)$ move two elements in common, they are equal. Their product is therefore the identity element of $S_n$, and we obtain it once for each of the ${n \choose 2}$ transpositions in $S_n$. The most important feature of Equation \ref{eqn:FH_alg_example} is that the coefficients depend polynomially on $n$, the parameter of the symmetric group $S_n$. 
\begin{definition}
The \emph{reduced cycle type} of an element of a symmetric group is the partition obtained by subtracting $1$ from each part of the cycle type, and ignoring any resulting parts of size zero. 
\end{definition}
\noindent
For example, the reduced cycle type of the identity is the empty partition, which we denote $\varnothing$. The reduced cycle type of a transposition is $(1)$. Note that the cycle type of an element of $S_n$ can be recovered from the reduced cycle type, provided that $n$ is known. The condition for $S_n$ to have an element of reduced cycle type $\mu$ is $n \geq |\mu| + l(\mu)$. The upshot of working with reduced cycle types is that if we let $X_\mu$ be the sum of all elements of reduced cycle type $\mu$, Equation \ref{eqn:FH_alg_example} becomes
\begin{equation} \label{precise_FH_alg_example}
X_{(1)}^2 = 2X_{(1,1)} + 3X_{(2)} + {n \choose 2}X_{\varnothing}.
\end{equation}
This equation is valid in $Z(\mathbb{Z}S_n)$ for any $n \in \mathbb{Z}_{\geq 0}$, provided $X_\mu$ is interpreted as zero if there are no elements of reduced cycle type $\mu$ in $S_n$ (in that case $X_\mu$ as we have defined it would be an empty sum). It turns out that any product of the $X_\mu$ decomposes as a linear combination with coefficients that are polynomial in $n$.

\begin{definition}
The \emph{ring of integer-valued polynomials}, $\mathcal{R}$, is the subring of $\mathbb{Q}[t]$ consisting of elements $p(t)$ such that $p(m) \in \mathbb{Z}$ for all $m \in \mathbb{Z}$.
\end{definition}
\noindent
It is well known that $\mathcal{R}$ is free as a $\mathbb{Z}$-module, with basis ${t \choose r}$ for $r \in \mathbb{Z}_{\geq 0}$.

\begin{theorem}[Farahat-Higman, Theorem 2.2 \cite{FarahatHigman}] \label{thm:FH_alg_polys}
For any partitions $\lambda, \mu, \nu$, there exists a unique integer-valued polynomial $\phi_{\mu, \nu}^\lambda(n)$ such that the equation
\[
X_\mu X_\nu = \sum_\lambda \phi_{\mu, \nu}^\lambda(n)X_\lambda
\]
holds in $Z(\mathbb{Z}S_n)$ for all $n \in \mathbb{Z}_{\geq 0}$.
\end{theorem}
\noindent
Using this theorem we can depart from the setting of $Z(\mathbb{Z}S_n)$ for a specific $n$, and work with all $n$ at once. This comes at the price of working over $\mathcal{R}$.

\begin{definition}[Farahat-Higman, Section 2 \cite{FarahatHigman}] \label{def:FH_alg}
Let $\FH$ be the free $\mathcal{R}$-module with basis $K_\mu$ indexed by all partitions $\mu$. Define a bilinear multiplication on $\FH$ by
\[
K_\mu K_\nu = \sum_{\lambda} \phi_{\mu, \nu}^\lambda(t) K_\lambda
\]
where $\phi_{\mu, \nu}^\lambda(t)$ are the polynomials from Theorem \ref{thm:FH_alg_polys}, viewed as elements of the base ring $\mathcal{R}$. We call $\FH$ the Farahat-Higman algebra.
\end{definition}
\noindent
Because the equations in Theorem \ref{thm:FH_alg_polys} hold in $Z(\mathbb{Z}S_n)$ for any $n$, we can return from $\FH$ to $Z(\mathbb{Z}S_n)$ via an appropriate specialisation homomorphism that sends $K_\mu$ to $X_\mu$ and evaluates the polynomial coefficients at $n$. It is then routine to verify that $\FH$ is a commutative, associative, unital $\mathcal{R}$-algebra; this reduces to considering the respective properties in $Z(\mathbb{Z}S_n)$ for sufficiently large $n$. 

\begin{theorem}[Farahat-Higman, Theorem 2.4 \cite{FarahatHigman}] \label{thm:FH_spec_hom}
For each $n \in \mathbb{Z}_{\geq 0}$ there is a surjective ring homomorphism $\Phi_n: \FH \to Z(\mathbb{Z}S_n)$ defined by
\[
\Phi_n \left( \sum_{\mu} a_\mu(t) K_\mu \right) = \sum_{\mu} a_\mu(n) X_\mu,
\]
where $a_\mu(t) \in \mathcal{R}$, and $X_\mu$ is the sum of all elements of $S_n$ of reduced cycle type $\mu$ (or zero, if there are no such elements).
\end{theorem}
\begin{remark}
Although we do not discuss it in this paper, there is a different construction and perspective of the algebra $\FH$ provided by Ivanov and Kerov \cite{IvanovKerov}. Their approach transparently explains the above properties of $\FH$. 
\end{remark}
\noindent
The original motivation of Farahat and Higman for constructing $\FH$ was to provide a simple proof of Nakayama's Conjecture (a theorem which describes the $p$-blocks of symmetric groups, originally proved by Brauer and Robinson \cite{BrauerRobinson}). In the remainder of this section, we show that $\FH = \mathcal{R} \otimes \Lambda$, explain Jucys-Murphy elements, and discuss the theory of content evaluation character formulae for $S_n$. This will allow us to succinctly describe Farahat-Higman's proof of Nakayama's Conjecture.

\subsection{Isomorphism with Symmetric Functions}
The main result of Farahat and Higman's paper is that $\FH$ is generated by a particular set of elements. With some care, we can infer that $\FH$ is isomorphic to the ring of symmetric functions with coefficients in $\mathcal{R}$ (Theorem \ref{thm:FH_lambda_iso}).
\begin{theorem}[Farahat-Higman, Theorem 2.5 \cite{FarahatHigman}] \label{thm:FH_generation_thm}
As an $\mathcal{R}$-algebra, $\FH$ is generated by the elements
\[
g_n = \sum_{\mu \vdash n} K_\mu
\]
where $n \in \mathbb{Z}_{>0}$.
\end{theorem}
\noindent
Recall that $\Lambda = \mathbb{Z}[e_1, e_2, \ldots]$. It is graded:
\[
\Lambda = \bigoplus_{k \geq 0} \Lambda^k
\]
where $e_r$ has degree $r$. In particular, if $\nu = (\nu_1, \nu_2, \ldots, \nu_l)$ is a partition, then $e_\nu = e_{\nu_1} e_{\nu_2} \cdots e_{\nu_l}$ is in degree $|\nu|$, and a basis of $\Lambda_n$ as a $\mathbb{Z}$-module is given by the set of $e_\nu$ as $\nu$ varies across partitions of $n$. The grading on $\Lambda$ induces a filtration 
\[
\Lambda = \bigcup_{i \geq 0} \Lambda^{\leq i},
\]
where
\[
\Lambda^{\leq i} = \bigoplus_{k \leq i} \Lambda^k,
\]
and $\Lambda^{\leq i}$ has $\mathbb{Z}$-basis consisting of $e_\nu$ where $|\nu| \leq i$.
\newline \newline \noindent
In fact $\FH$ is also a filtered $\mathcal{R}$-algebra,
\[
\FH = \bigcup_{i \geq 0} \mathcal{F}^i
\]
where $\mathcal{F}^i$ is the $\mathcal{R}$-submodule of $\FH$ spanned by $K_\mu$ with $|\mu| \leq i$ (see Lemma 3.9 of \cite{FarahatHigman}). This filtration may be interpreted as follows. A $k$-cycle in a symmetric group may be written as a product of $k-1$ transpositions (and no fewer), for example:
\[
(i_1, i_2, i_3, \ldots, i_k) = (i_1, i_2)(i_2, i_3)\cdots(i_{k-1},i_k).
\]
For an arbitrary permutation of cycle type $\nu = (\nu_1, \nu_2, \ldots, \nu_l)$, the number of transpositions needed is
\[
\sum_i (\nu_i - 1),
\]
which is precisely the size of the corresponding reduced cycle type $(\nu_1-1, \nu_2-1, \ldots, \nu_k-1)$. So $\mathcal{F}^i$ may be seen as filtering permutations according to how many transpositions are needed to construct them.

\begin{theorem} \label{thm:FH_lambda_iso}
There is an isomorphism $\Psi: \mathcal{R} \otimes_{\mathbb{Z}} \Lambda \to \FH$ of filtered $\mathcal{R}$-algebras defined by
\[
\Psi(e_n) = g_n = \sum_{\mu \vdash n} K_\mu.
\]
\end{theorem}
\noindent
We give two proofs. The first proof is intended to be in the spirit of the original work of Farahat and Higman. The second proof relies on the later work of Jucys, and it is this second proof that will generalise to the wreath-product setting.
\begin{proof}[Proof 1]
The homomorphism $\Psi$ is well defined because $\Lambda$ is a free polynomial algebra, so there are no relations that need to be checked. Since $\Psi(e_i) = g_i \in \mathcal{F}^i$, it is immediate that $\Psi$ respects the filtrations on both spaces. Additionally, Theorem \ref{thm:FH_generation_thm} shows $\Psi$ is surjective. However, the proof in \cite{FarahatHigman} proceeds by showing that \[
\mathcal{F}^i = \mathcal{R}g_i + \sum_{\substack{j+k = i \\ j,k \geq 1}}\mathcal{F}^j \cdot \mathcal{F}^k.
\]
From this, it follows by induction on $i$ that the restriction of $\Psi$ to a map from the $i$-th filtered component of $\Lambda$ to the $i$-th filtered component of $\FH$ is a surjection:
\begin{eqnarray*}
\mathcal{F}^i &=& \mathcal{R}g_i + \sum_{\substack{j+k = i \\ j,k \geq 1}}\mathcal{F}^j \cdot \mathcal{F}^k \\
&=& \mathcal{R}\Psi(e_i) + \sum_{\substack{j+k = i \\ j,k \geq 1}}\Psi(\Lambda^{\leq j}) \cdot \Psi(\Lambda^{\leq k}) \\
&\subseteq& \Psi(\Lambda^{\leq i}).
\end{eqnarray*}
We may pass to the fraction field of $\mathcal{R}$, namely the field of rational functions $\mathbb{Q}(t)$. Now $\Psi$ restricts to a surjection of vector spaces:
\[
\left.\Psi\right|_{\Lambda^{\leq i}}:
\mathbb{Q}(t) \otimes_{\mathbb{Z}} \Lambda^{\leq i} \to \mathbb{Q}(t) \otimes_{\mathcal{R}} \mathcal{F}^i.
\]
Since both spaces have the same dimension (each has a basis indexed by partitions of size at most $i$), standard linear algebra shows that since the linear map $\left.\Psi\right|_{\Lambda^{\leq_i}}$ is surjective, it must also be injective. This in turn implies that $\Psi$ is injective.
\end{proof}
\begin{proof}[Proof 2]
By Theorem \ref{thm:elem_JM} below, $\Phi_n(\Psi(e_r)) \in Z(\mathbb{Z}S_n)$ may be interpreted as the evaluation of the elementary symmetric function $e_r$ at the Jucys-Murphy elements $L_1, \ldots, L_n$. By Proposition \ref{murphy_prop} below, applying the same operation to the monomial symmetric functions $m_\mu$, we get $X_\mu$ plus terms lower in a certain partial order. Since $\Phi_n(\Psi(m_\mu))$ is equal to $X_\mu$ plus lower-order terms, $\Psi(m_\mu)$ is equal to $K_\mu$ plus lower-order terms. Since the $m_\mu$ and $K_\mu$ are $\mathcal{R}$-bases of the respective spaces, this shows that $\Psi$ is an isomorphism.
\end{proof}
\noindent
Theorem \ref{thm:FH_lambda_iso} is essentially a strengthening of Theorem 3.1 of \cite{CorteelGoupilSchaeffer}, the difference being that they work rationally (i.e. over $\mathbb{Q} \otimes_{\mathbb{Z}} \mathcal{R} = \mathbb{Q}[t]$), while we work integrally (i.e. over $\mathcal{R}$). This will be essential when we consider modular representation theory in Section \ref{applications_section}.
\newline \newline \noindent
At this point it may be unclear why we choose to work with the ring of symmetric functions when the only property we have used is that it is a free polynomial algebra with one generator in each positive degree. In the next subsection, we will summarise the theory of Jucys-Murphy elements, which will allow us to interpret the isomorphism $\Psi$ as ``Jucys-Murphy evaluation'' of symmetric functions. This will give us formulae for central characters. To that end, we will need the following definition.
\begin{definition} \label{char_sym_fun_def}
For a partition $\mu$, the \emph{character symmetric function} is $f_\mu = \Psi^{-1}(K_\mu)$ (it is an element of $\mathcal{R} \otimes_{\mathbb{Z}} \Lambda$).
\end{definition}
\begin{example} \label{char_sym_fn_example}
By the definition of $\Psi$, $\Psi(e_1) = K_{(1)}$ and $\Psi(e_2) = K_{(1,1)} + K_{(2)}$. Additionally, Equation \ref{precise_FH_alg_example} shows that 
\[
\Psi(e_1^2) = 2 K_{(1,1)} + 3K_{(2)} + {n \choose 2}.
\]
From this we conclude that $f_{(1)} = e_1$, $f_{(2)} = e_1^2 - 2e_2 - {n \choose 2}$, and $f_{(1,1)} = 3e_2 - e_1^2 + {n \choose 2}$.
\end{example}

\noindent
One may consult Section 5.4 of \cite{CST} for an alternative exposition of character symmetric functions. Note however, that our indexing variable $\mu$ is a partition corresponding to a reduced cycle type, while some of the literature uses partitions corresponding to cycle types, but with parts of size $1$ removed.

\subsection{Jucys-Murphy Elements} The Jucys-Murphy (``JM") elements are a key part of what has come to be known as the Okounkov-Vershik approach to the representations of symmetric groups (\cite{OkounkovVershik} and \cite{Kleshchev} are both excellent references), some parts of which we briefly review before explaining the connection to the Farahat-Higman algebra.
\newline \newline \noindent
The JM elements $L_1, \ldots, L_n \in \mathbb{C}S_n$ are sums of certain transpositions:
\[
L_m = \sum_{1\leq i < m} (i,m).
\]
For example $L_1 = 0$, $L_2 = (1,2)$, and $L_3 = (1,3)+(2,3)$. It is well known that the JM elements pairwise commute, so if $P(x_1, \ldots, x_n) \in \mathbb{Z}[x_1, \ldots, x_n]$, the expression $P(L_1, \ldots, L_n) \in \mathbb{Z}S_n$ is unambiguous. Furthermore, if $P(x_1, \ldots, x_n)$ is symmetric in the $x_i$, then $P(L_1, \ldots, L_n)$ is known to be a central element of $\mathbb{Z}S_n$.

\begin{proposition}[Murphy, Theorem 1.9 \cite{Murphy}] \label{murphy_prop}
If $m_\mu$ is the monomial symmetric function, then $m_\mu(L_1, \ldots, L_n)$ is equal to $X_\mu$ plus a linear combination of $X_\nu$ such that either $|\nu| < |\mu|$ or $|\nu| = |\mu|$ and $l(\nu) < l(\mu)$.
\end{proposition}

\begin{proof}
Suppose that $g \in S_n$ has reduced cycle type $\mu$. Then multiplying $g$ by $(i,j)$ either merges two (possibly trivial) cycles if $i$ and $j$ are in distinct cycles of $g$, or splits an individual cycle into two cycles if $i$ and $j$ are in the same cycle of $g$. The merging of two cycles has the effect of increasing the size of the reduced cycle type by 1, while splitting subtracts 1 from the size. Since each JM element is a sum of transpositions, $m_\mu(L_1, \ldots, L_n)$ is a sum of products of $|\mu|$ transpositions. So to compute the leading order term, we only consider products that merge cycles at each step.
\newline \newline \noindent
We consider $L_j^r$. This is a sum of products
\[
(i_1, j)\cdots (i_r,j)
\]
of $r$ transpositions $(i_k,j)$ where the $i_k$ ($k=1,\ldots,r$) may be any numbers less than $j$. In order for the cycle to grow with each multiplication, it is necessary and sufficient that values of $i_k$ must be distinct. The result of such a product will be the $(r+1)$-cycle
\[
(i_r, i_{r-1}, \ldots, i_1, j).
\]
So we get every $(r+1)$ cycle with largest element $j$ exactly once because the elements of the cycle determine the transpositions involved in the product, and their order in the cycle determines the order of the transpositions in the product.
\newline \newline \noindent
Finally, we note that $m_\mu(L_1, \ldots, L_n)$ is a sum of products of $L_j^r$ where the exponents $r$ are the parts of $\mu$ in some order. By the previous paragraph, up to leading order, we get a product of $(r+1)$-cycles, and the length of the reduced cycle type is maximised when they do not intersect. In that case we get an element of reduced cycle type $\mu$, and moreover each such element arises exactly once because the length of the cycle whose largest element is $j$ must have been the exponent of $L_j$ in the monomial that gave rise to the cycle in consideration.
\end{proof}
\noindent
In the case where $\mu = (1^r)$ so that $m_\mu = e_r$, the above argument simplifies; it never happens that a transposition splits a cycle. This makes it possible to keep track of all the resulting permutations; we get the sum of all elements in $S_n$ with $n-r$ cycles.

\begin{theorem}[Jucys, Section 3 \cite{Jucys}] \label{thm:elem_JM}
If $e_r$ is the $r$-th elementary symmetric function, then
\[
e_r(L_1, \ldots, L_n) = \Phi_n(g_r) = \sum_{\mu \vdash r} X_\mu,
\]
where $X_\mu$ is the sum of elements of reduced cycle type $\mu$ in $S_n$.
\end{theorem}
\noindent
In view of Theorem \ref{thm:FH_spec_hom}, we have the following commutative diagram for each $n$,
\begin{equation}
\begin{tikzcd}
\mathcal{R} \otimes_{\mathbb{Z}} \Lambda    \arrow[r, "\Psi"]    \arrow[dr, "ev_n"]    &    \FH    \arrow[d, "\Phi_n"]\\
&    Z(\mathbb{Z}S_n)
\end{tikzcd}
\end{equation}
\noindent
where $ev_n: \mathcal{R} \otimes_{\mathbb{Z}} \Lambda \to Z(\mathbb{Z}S_n)$ evaluates an integer-valued polynomial at $n$ and evaluates a symmetric function at $L_1, \ldots, L_n$. To see that the diagram commutes, it suffices to check the generating elements $e_r$ of $\Lambda$, which is the statement of Theorem \ref{thm:elem_JM}. This suggests that $\FH$ can be thought of as being ``the ring of symmetric functions evaluated at JM elements''.
\newline \newline \noindent
Recall that the irreducible representations of $\mathbb{C}S_n$ are the Specht modules $S^\lambda$, indexed by partitions $\lambda$ of size $n$. The Specht module $S^\lambda$ possesses a Gelfand-Zetlin (``GZ'') basis $v_T$ indexed by standard Young tableaux $T$ of shape $\lambda$. A very special property of the GZ basis is that it is diagonal for the action of the JM elements, and the eigenvalues are given by the content of boxes in $T$:
\[
L_i v_T = c_T(i) v_T.
\]
\noindent
Now suppose that $P(x_1, \ldots, x_n)$ is a symmetric polynomial. Evaluating at the JM elements and applying it to $v_T$, we obtain
\[
P(L_1, \ldots, L_n)v_T = P(c_T(1), \ldots, c_T(n))v_T.
\]
Since $P$ is symmetric, $P(c_T(1), \ldots, c_T(n))$ does not depend on the order of the content of $T$. But for any $T$ of shape $\lambda$, the multiset of $c_T(i)$ is precisely the content of the boxes in $\lambda$, in particular, it is the same for every $T$ of shape $\lambda$. This implies that $P(L_1, \ldots, L_n)$ acts by the same scalar on any $v_T$ of fixed shape $\lambda$. This implies that $P(L_1, \ldots, L_n) \in \mathbb{Z}S_n$ acts by scalar multiplication on each irreducible representation of the symmetric group, and is therefore a central element of $\mathbb{Z}S_n$ as stated earlier.

\begin{theorem} \label{sym_gp_eval_thm}
The central character for $S_n$, $\omega_\mu^\lambda$ ($\mu$ defines a conjugacy class via reduced cycle type), is equal to the character symmetric function $f_\mu$ evaluated at the content of the partition $\lambda$, with the integer-valued polynomial variable $t$ evaluated at $n$.
\end{theorem}
\begin{proof}
By definition $\omega_\mu^\lambda$ is the scalar by which $X_\mu$ acts on $S^\lambda$. To calculate the scalar, we may choose an arbitrary GZ basis vector $v_T$ and act on it by $ev_n(f_\mu) = X_\mu$. But $ev_n$ evaluates elements of $\mathcal{R}$ at $t=n$ and evaluates the symmetric function variables at JM elements, which act on $v_T$ by the contents of $\lambda$.
\end{proof}

\begin{example}
We know from Example \ref{char_sym_fn_example} that $f_{(1)} = e_1$. As $(1)$ is the reduced cycle type of a transposition, we get that sum of all transpositions acts on the Specht module $S^\lambda$ as multiplication by the sum of the contents of the boxes in the Young diagram $\lambda$. For a box $\square$ in the diagram of $\lambda$, let $row(\square)$ and $col(\square)$ be the number of the row and column containing $\square$ respectively. Then the sum of the contents of $\lambda$ is
\[
\sum_{\square \in \lambda} col(\square) - row(\square) = \sum_{\square \in \lambda} col(\square) - \sum_{\square \in \lambda} row(\square) = \sum_i {\lambda_i +1 \choose 2} - \sum_{j} {\lambda_j^\prime +1 \choose 2},
\]
where $\lambda_j^\prime$ are the lengths of the columns in the diagram of $\lambda$ (equivalently, the parts of the partition dual to $\lambda$). Here we have used the identity $1 + 2 + \cdots + m = {m +1 \choose 2}$ to sum each row/column. We have recovered the celebrated Frobenius formula \cite{Frobenius}, see also Example 7 of Section 1.7 of \cite{Macdonald}. 
\end{example}

\noindent
We are now able to give a proof of Nakayama's Conjecture, which is a characterisation of the $p$-blocks of $S_n$. Although it was first proved by Brauer and Robinson \cite{BrauerRobinson}, we take the simpler approach of Farahat and Higman.
\begin{theorem}[Farahat-Higman, \cite{FarahatHigman}] \label{nakayama_conjecture}
Two irreducible representations of $\mathbb{C}S_n$ are in the same $p$-block if and only if they are labelled by partitions with the same $p$-core.
\end{theorem}
\begin{proof}
The definition of $p$-block makes reference to a $p$-modular system $(K, \mathcal{O}, k)$. Proposition \ref{block_prop} provides a way to determine the blocks in terms of central characters. Note that since $\mathcal{O}/(\pi) = k$ has characteristic $p$, $p \in \mathcal{O}$ is contained in the ideal $(\pi)$. Suppose that $a, b \in \mathbb{Z} \subseteq \mathcal{O}$ are integers viewed as elements of $\mathcal{O}$. Then if $p | a-b$, certainly $\pi | a-b$. However, if $\pi | a-b$, then $a-b$ is an integer divisible by $\pi$. But if $\pi$ divides an integer $m$ coprime to $p$, then by B\'{e}zout's identity, $\pi$ divides $\gcd(m,p)=1$, which is a contradiction as $\pi$ is not a unit. Thus $\pi | a-b$ implies $p | a-b$, and we conclude that the $a$ and $b$ are congruent modulo $\pi$ if and only if the are congruent modulo $p$.
\newline \newline \noindent
Theorem \ref{sym_gp_eval_thm} shows that the central characters of $S_n$ are integers because the value of an element of $\mathcal{R}$ at $n$ is by definition an integer, and evaluating symmetric polynomials with integer coefficients at integers (contents of $\lambda$) will give integers. So we are left to determine them modulo $\pi$, or equivalently, modulo $p$. The centre of the group algebra of $S_n$ is $ev_n(\mathcal{R} \otimes \Lambda)$. Because $\mathcal{R} \otimes \Lambda$ is generated by $\mathcal{R}$ and the elementary symmetric functions $e_i$, it is enough to consider the action of $ev_n(\mathcal{R})$ and $ev_n(e_i)$ on Specht modules $S^\lambda$. The action of $ev_n(\mathcal{R})$ is independent of the partition $\lambda$. To understand the action of the elementary symmetric polynomials, we assemble them into a generating function:
\[
\sum_i ev_n(e_i) t^i = 
\sum_i e_i(L_1, \ldots, L_n)t^i = \prod_{i=1}^n (1 + L_i t),
\]
which acts on a GZ basis vector $v_T$ ($T$ is a standard Young tableau of shape $\lambda$) by the scalar
\[
\prod_{i=1}^n (1 + c_T(i) t).
\]
By the unique factorisation of polynomials in $k[t]$, this generating function determines, and is determined by, the content of $\lambda$ viewed as elements of $k$, i.e. taken modulo $\mathrm{char}(k)=p$. Two irreducibles are in the same $p$-block if and only if they have the same central characters, which holds if and only if they are labelled by partitions with the same content modulo $p$, which holds if and only if the partitions have the same $p$-core. 
\end{proof}

\begin{remark}
Some authors (e.g. \cite{Wang1} and Example 25 of Section 1.7 of \cite{Macdonald}) consider the associated graded algebra of $\FH$ (with respect to the filtration defined immediately before Theorem \ref{thm:FH_lambda_iso}). One advantage of this is that the structure constants in the associated graded algebra are integers, rather than arbitrary elements of $\mathcal{R}$, so one may avoid working over $\mathcal{R}$ altogether. One disadvantage is that the maps $ev_n: \FH \to Z(\mathbb{Z}S_n)$ become maps to an associated graded version of the centre of the group algebra (where $X_\mu$ is in degree $|\mu|$). Although this obstructs applications to modular representation theory, it turns out to be the right thing to do in the setting of Hilbert schemes. In fact, this associated graded version of $Z(\mathbb{Z}S_n)$ is isomorphic to $H^*(\mathrm{Hilb}^n(\mathbb{A}_{\mathbb{C}}^2), \mathbb{Z})$, the cohomology ring (with $\mathbb{Z}$ coefficients) of the Hilbert scheme of $n$ points in the plane. This was shown in \cite{LehnSorger}. We will not discuss Hilbert schemes any further in this paper.
\end{remark}

\section{\texorpdfstring{$\Gamma_*$}{Gamma}-Weighted Symmetric Functions} \label{LambdaG_section}
\noindent
We now generalise the construction of the ring of symmetric functions in a way that incorporates $\Gamma_*$, which will define a ring $\LambdaG$ that will play a central role in this paper. Let $Q$ be the subring of $\mathbb{Z}\Gamma_*[x]$ consisting of polynomials whose constant term is a multiple of the identity (rather than an arbitrary element of $\mathbb{Z}\Gamma_*$). Consider the $n$-th tensor power (over $\mathbb{Z}$) of $Q$, which has an action of $S_n$ by permutation of tensor factors. As shorthand, for $c \in \mathbb{Z}\Gamma_*$ we write 
\[
x_i^r(c) = 1^{\otimes (i-1)} \otimes cx^r \otimes 1^{\otimes (n-i)}.
\]
This means that $x_i^r(c) x_i^s(c^\prime) = x_i^{r+s}(cc^\prime)$, and that $x_i^r(c) + x_i^r(c^\prime) = x_i^r(c+c^\prime)$.
Since $Q$ has a grading (inherited from $\mathbb{Z}\Gamma_*[x]$), there is a grading on the ring of $S_n$ invariants of $Q^{\otimes n}$. We write $\Lambda_n(\Gamma_*)$ for the $S_n$-invariants of $Q^{\otimes n}$, and $\Lambda_n^k(\Gamma_*)$ for the degree $k$ component of $\Lambda_n(\Gamma_*)$. As before, for $m>n$ there are homomorphisms $\rho_{m,n}: \Lambda_{m}^k(\Gamma) \to \Lambda_{n}^k(\Gamma)$ which evaluate evaluate the polynomial variables of all tensor factors past the $n$-th at zero. These define an inverse system, and we write
\[
\Lambda^k(\Gamma_*) = \varprojlim \Lambda_n^k(\Gamma_*).
\]
\begin{remark}
It may seem unnatural to work with the ring $Q$ rather than the full polynomial ring $\mathbb{Z}\Gamma_*[x]$. The reason we do this is that evaluating elements of $Q$ at zero yields an element of $\mathbb{Z}$ rather than $\mathbb{Z}\Gamma_*$, so the codomain of $\rho_{m,n}$ is
\[
\Lambda_n^k(\Gamma_*) \otimes \mathbb{Z}^{\otimes (m-n)} = \Lambda_n^k(\Gamma_*),
\]
rather than
\[
\Lambda_n^k(\Gamma_*) \otimes (\mathbb{Z}\Gamma_*)^{\otimes (m-n)},
\]
which, being different from $\Lambda_n^k(\Gamma_*)$, would not allow us to construct an inverse system. Later we will account for these ``missing'' constant terms (see Theorem \ref{r_gamma_hom_thm}), using the ring $\RGamma$ which is defined in Section \ref{rgamma_section}.
\end{remark}
\noindent
The ring $Q$ has a basis consisting of elements of the form $c x^j$ where $c \in \Gamma_*$ and $j \in \mathbb{Z}_{\geq 0}$ (we require $c = 1$ if $j=0$). Thus the ring $Q^{\otimes n}$ has a basis consisting of pure tensors in this basis of $Q$. We refer to such a pure tensor as a \emph{$\Gamma_*$-weighted monomial} and note that the $S_n$ action sends $\Gamma_*$-weighted monomials to other $\Gamma_*$-weighted monomials.
\begin{definition}
Let $\bs \lambda$ be a multipartition indexed by $\Gamma_*$. The \emph{$\Gamma_*$-weighted monomial symmetric polynomial}, $m_{\bs \lambda}(x_1, \ldots, x_n) \in Q^{\otimes n}$ is the sum of all $\Gamma_*$-weighted monomials in $Q^{\otimes n}$ that contain $c x^j$ with $j \geq 1$ as a tensor factor exactly $m_j(\bs \lambda(c))$ times. There is no restriction on the number of times $1$ may appear as a tensor factor.
\end{definition}
\noindent
It is immediate that the degree of $m_{\bs\lambda}(x_1, \ldots, x_n)$ is $|\bs\lambda|$. Since the $\Gamma_*$-weighted monomial symmetric polynomials are orbit sums for the $S_n$ action on our basis of $Q^{\otimes n}$, it follows that $m_{\bs\lambda}(x_1, \ldots, x_n)$ with $|\bs\lambda|=k$ span $\Lambda_{n}^k(\Gamma_*)$, and the nonzero ones form a basis of this space. Moreover $m_{\bs\lambda}(x_1, \ldots, x_n)$ is nonzero as soon as there are enough tensor factors to accommodate all the basis vectors prescribed by $\bs \lambda$, i.e. as soon as $n \geq l(\bs\lambda)$. Finally, we observe that
\[
\rho_{m,n}(m_{\bs\lambda}(x_1,\ldots, x_m)) = m_{\bs\lambda}(x_1,\ldots, x_n),
\]
so these elements define an element of the inverse limit $\Lambda^k(\Gamma_*)$.
\begin{definition}
The \emph{$\Gamma_*$-weighted monomial symmetric function} $m_{\bs\lambda}$ is the element of $\Lambda^k(\Gamma_*)$ defined by the sequence of elements $m_{\bs\lambda}(x_1, \ldots, x_n)$ as $n$ varies (here $k = |\bs\lambda|$).
\end{definition}
\noindent
It now follows that $\Lambda^k(\Gamma_*)$ is free as a $\mathbb{Z}$-module with basis $m_{\bs\lambda}$ indexed by all multipartitions $\bs\lambda$ of size $k$.
\begin{definition}
The ring of $\Gamma_*$-weighted symmetric functions is
\[
\LambdaG = \bigoplus_{k=0}^{\infty} \Lambda^k(\Gamma_*).
\]
\end{definition}
\noindent
Similarly to the case of ordinary symmetric functions, the fact that evaluating a polynomial variable at zero is a ring homomorphism implies that $\LambdaG$ inherits the structure of a graded ring. In fact, if $\Gamma$ is the trivial group, then $\LambdaG = \Lambda$. Just as for $\Lambda$, formal properties of inverse limits automatically give us ring homomorphisms
\[
\LambdaG \to \Lambda_n(\Gamma_*).
\]

\begin{example} \label{degree_2_example}
We demonstrate how to multiply two degree 1 elements of $\LambdaG$. Suppose that $c_r \in \Gamma_*$ is a conjugacy class. Then $m_{(1)_{c_r}}$ is the element of $\Lambda^1(\Gamma)$ defined by the sequence of elements $\sum_{i=1}^n x_i(c_r) \in \Lambda_n^1(\Gamma_*)$. So we must express products of such elements in terms of $\Gamma_*$-weighted monomial symmetric polynomials $m_{\bs\lambda}(x_1, \ldots, x_n)$. We find that for any number of variables,
\begin{eqnarray*}
m_{(1)_{c_r}}^2 &=& \left( \sum_i x_i(c_r) \right)^2 \\
&=& \sum_i x_i^2(c_r^2) +  \sum_{i \neq j} x_i(c_r) x_j(c_r) \\
&=& \sum_i \sum_{c_s \in \Gamma_*} A_{r,r}^s x_i^2(c_s) + 2 \sum_{i<j} x_i(c_r) x_j(c_r) \\
&=& \sum_{c_s \in \Gamma_*} A_{r,r}^s m_{(2)_{c_s}} + 2 m_{(1,1)_{c_r}},
\end{eqnarray*}
and therefore the equality between the first and last quantities may be interpreted as holding in $\Lambda^2(\Gamma_*)$. Similarly, if $c_r,c_s \in \Gamma_*$ are distinct conjugacy classes, then 
\begin{eqnarray*}
m_{(1)_x}^2 &=& \left( \sum_i x_i(c_r) \right)\left( \sum_j x_j(c_s) \right) \\
&=& \sum_i x_i^2(c_rc_s) +  \sum_{i \neq j} x_i(c_r) x_j(c_s) \\
&=& \sum_{c_t \in \Gamma_*} A_{r,s}^t m_{(2)_t} + 2 m_{(1)_{c_r}(1)_{c_s}}.
\end{eqnarray*}
\end{example}
\noindent
Analogously to how $\mathbb{Z}$-linear combinations of conjugacy-class sums and $\mathbb{Z}$-linear combinations of central idempotents in $\mathbb{C}\Gamma$ define different integral forms of $\mathbb{C}\Gamma_*$, the ring $\LambdaG$ is a different integral form of the $|\Gamma^*|$-th tensor power of $\Lambda$.
\begin{proposition} \label{complex_bas_change_prop}
We have
\[
\mathbb{C} \otimes \LambdaG = \mathbb{C} \otimes \Lambda^{\otimes |\Gamma^*|}.
\]
\end{proposition}
\begin{proof}
We use the fact that $\mathbb{C} \otimes \mathbb{Z}\Gamma_* = \mathbb{C} \Gamma_* = \mathbb{C}^{\Gamma^*}$, where standard basis vectors in $\mathbb{C}^{\Gamma^*}$ are the orthogonal central idempotents $e_\chi$ in the group algebra $\mathbb{C}\Gamma$ associated to irreducible representations $\chi \in \Gamma^*$. Then $e_\chi e_{\psi} = \delta_{\chi, \psi} e_{\chi}$. Working over $\mathbb{C}$ allows us to define elements $m_{\bs\lambda}^{irr} \in \LambdaG$ similar to the $\Gamma_*$-weighted monomial symmetric functions, but instead of taking sums of pure tensors in the basis $\Gamma_*$ of $\mathbb{C}\Gamma_*$, we use the basis $e_\chi$. So $x_i^r (e_\chi)$ form a basis of (the positive degree part of) $\mathbb{C} \otimes Q$, and $m_{\bs\lambda}^{irr}$ is constructed by taking a sum of a $S_n$-orbit of a pure tensor of elements of this basis. Hence $\bs\lambda \in \mathcal{P}(\Gamma^*)$ is indexed by $\Gamma^*$ rather than $\Gamma_*$. The upshot of this is that since 
\[
x_i^r(e_\chi) x_i^s( e_\psi) = \delta_{\chi, \psi} x_i^{r+s}(e_\chi),
\]
where $\delta_{\chi, \psi}$ is the Kronecker delta, terms corresponding to different $\chi \in \Gamma^*$ do not interact. More precisely, suppose that $\bs \lambda$ is a multipartition and let ${\bs \lambda}_\chi = {\bs \lambda}(\chi)_{\chi}$ be the multipartition concentrated in type $\chi$ taking the value $\bs \lambda(\chi)$. Then we have
\[
m_{\bs\lambda}^{irr} = \prod_{\chi \in \Gamma^*} m_{{\bs\lambda}_\chi}^{irr}.
\]
This is because each monomial in $m_{\bs\lambda}^{irr}$ arises uniquely as a product of monomials from each $m_{{\bs\lambda}_\chi}^{irr}$, each involving distinct variables. Conversely, any product of monomials from $m_{{\bs\lambda}_\chi}^{irr}$ involving the same variable for different $\chi$ is zero because $x_i^r(e_\chi) x_i^s( e_\psi) = \delta_{\chi, \psi} x_i^{r+s}(e_\chi)$. This immediately shows that
\[
\mathbb{C} \otimes \LambdaG = \bigotimes_{\chi \in \Gamma^*} \Lambda(\chi),
\]
where $\Lambda(\chi)$ has $\mathbb{C}$-basis $m_{\bs \lambda}^{irr}$, where $\bs\lambda$ is concentrated in type $\chi$. Finally, it remains to observe that $\Lambda(\chi)$ is a ring isomorphic to $\mathbb{C} \otimes \Lambda$ via the $\mathbb{C}$-linear map taking $m_{\bs\lambda}^{irr}$ to $m_{\bs\lambda(\chi)}$. This is clearly a bijection. To see that it respects multiplication note that the multiplication in $\Lambda(\chi)$ is determined by the following rule for products in $\mathbb{C} \otimes Q$: $x_i^r(e_\chi) x_i^s( e_\chi) = x_i^{r+s}(e_\chi)$. In $\Lambda$ the corresponding relation reads $x_i^r \cdot x_i^s = x_i^{r+s}$.
\end{proof}

\begin{example}
Suppose that $\Gamma = C_2 = \{1, \gamma\}$ is the cyclic group of order two (so $\gamma^2=1$). We show that $\Lambda({C_2}_*)$ is not isomorphic to $\Lambda \otimes \Lambda$ as a graded ring by considering the module of indecomposables. Suppose that $A = \bigoplus_{i=0}^\infty A_i$ is a graded algebra and $I = \bigoplus_{i=1}^\infty A_i$ is the corresponding augmentation ideal. The \emph{module of indecomposables} is defined to be $I/I^2$, which is a (graded) module for $A/I$. We compute the degree two component of $I/I^2$ when $A$ is either $\Lambda({C_2}_*)$ or $\Lambda \otimes \Lambda$.
\newline \newline \noindent
Recall that $\Lambda = \mathbb{Z}[e_1, e_2, \ldots]$. If $A = \Lambda \otimes \Lambda$, then the degree $1$ component of $A$ has basis $e_1 \otimes 1, 1 \otimes e_1$, while the degree $2$ component has basis $e_2 \otimes 1, e_1^2 \otimes 1, e_1 \otimes e_1, 1 \otimes e_2, 1 \otimes e_1^2$. Then the degree 2 component of $I^2$ has basis $e_1^2 \otimes 1, e_1 \otimes e_1, 1 \otimes e_1^2$. We conclude that the degree $2$ component of the module of indecomposables of $A$ is $\mathbb{Z}^2$.
\newline \newline \noindent
Analogously, for $A = \Lambda({C_2}_*)$, the degree $1$ component has basis $m_{(1)_1}, m_{(1)_\gamma}$, while the degree $2$ component has basis $m_{(1,1)_1}, m_{(2)_1}, m_{(1)_1 (1)_\gamma}, m_{(1,1)_\gamma}, m_{(2)_\gamma}$. The degree $2$ component of $I^2$ is spanned by 
\begin{eqnarray*}
m_{(1)_1}^2 &=& m_{(2)_1} + 2m_{(1,1)_1} \\
m_{(1)_1}m_{(1)_\gamma} &=&  m_{(2)_\gamma} + m_{(1)_1 (1)_\gamma} \\
m_{(1)_\gamma}^2 &=& m_{(2)_1} + 2 m_{(1,1)_\gamma},
\end{eqnarray*}
which we computed in Example \ref{degree_2_example}. From this we see that the degree $2$ component of $I/I^2$ is isomorphic to $\mathbb{Z}^{2} \oplus \mathbb{Z}/2\mathbb{Z}$. We conclude that $\Lambda({C_2}_*) \neq \Lambda^{\otimes 2}$. We discuss the algebraic structure of $\LambdaG$ further in the appendix.
\end{example}

\section{Wreath-Product Farahat-Higman Algebras} \label{rgamma_section}
\noindent
We now extend the results of Section \ref{FH_theory_section} from the symmetric group $S_n$ to the wreath products $\Gamma \wr S_n = \Gamma ^n \rtimes S_n$, where $\Gamma$ is a finite group. There is a wreath-product version of the Farahat-Higman algebra, which is also related to the ring of symmetric functions. The homomorphisms $\Phi_n$, $ev_n$, $\Psi$ all generalise to the wreath-product setting. We use the same notation for the wreath-product versions of these maps because taking $\Gamma$ to be the trivial group (so that $\Gamma \wr S_n = S_n$), we recover the maps from the Section \ref{FH_theory_section}.
\newline \newline \noindent
Recall that the conjugacy classes of $\Gamma \wr S_n$ correspond to multipartitions of total size $n$ indexed by $\Gamma_*$. We use the boldface Greek letters $\bs \lambda, \bs \mu, \bs \nu$ to indicate multipartitions. Thus $\bs{\mu}(c)$ means the partition in $\bs \mu$ indexed by $c \in \Gamma_*$. Similarly to symmetric groups, we have a notion of reduced cycle type. 
\begin{definition}[Wang, Subsection 2.3 \cite{Wang1}]
The \emph{partially-reduced cycle type} of an element of $\Gamma \wr S_n$ of cycle type $\bs{\mu}$ is the multipartition obtained by subtracting $1$ from each part of $\bs \mu(1)$, and ignoring any resulting parts of size zero. 
\end{definition}
\noindent
This is the same as the case of symmetric groups, but applied only to the partition labelled by the identity conjugacy class. The group $\Gamma \wr S_n$ contains an element of partially-reduced cycle type $\bs \mu$ if and only if $n \geq |\bs \mu| + l(\bs \mu(1))$. Later in this section we will need to introduce the notion of \emph{fully-reduced cycle type} which will involve reducing each partition in $\bs \mu$ (this is why we do not use the terminology of ``modified type'' from \cite{Wang1}).
\newline \newline \noindent
The centre of the integral group ring of $\Gamma \wr S_n$ has a $\mathbb{Z}$-basis consisting of conjugacy class sums. We let $X_{\bs \mu}$ be the sum of all elements of partially-reduced cycle type $\bs \mu$ (which if zero if there are no such elements). Then we have a direct analogue of Theorem \ref{thm:FH_alg_polys}.

\begin{theorem}[Wang, Theorem 2.13 \cite{Wang1}] \label{FH_g_alg_polys_thm}
For any multipartitions $\bs \lambda, \bs\mu, \bs\nu$, there exists a unique integer-valued polynomial $\phi_{\bs\mu, \bs\nu}^{\bs\lambda}(n)$ such that the equation
\[
X_{\bs\mu} X_{\bs\nu} = \sum_{\bs\lambda} \phi_{\bs\mu, \bs\nu}^{\bs\lambda}(n)X_{\bs\lambda}
\]
holds in $Z(\mathbb{Z} \Gamma \wr S_n)$ for all $n \in \mathbb{Z}_{\geq 0}$.
\end{theorem}
\noindent
This allows us to define the an algebra analogous to $\FH$, but for wreath products.


\begin{definition} 
Let $\FHG$ be the free $\mathcal{R}$-module with basis $K_{\bs \mu}$ indexed by all multipartitions $\bs \mu \in \mathcal{P}(\Gamma_*)$. Define a bilinear multiplication on $\FHG$ by
\[
K_{\bs \mu} K_{\bs \nu} = \sum_{\lambda} \phi_{\bs\mu, \bs\nu}^{\bs\lambda}(t) K_{\bs \lambda}
\]
where $\phi_{\bs\mu, \bs\nu}^{\bs\lambda}(t)$ are the polynomials from Theorem \ref{FH_g_alg_polys_thm}, viewed as elements of the base ring $\mathcal{R}$.
\end{definition}
\noindent
As in the case of $\FH$, $\FHG$ is a commutative, associative, unital $\mathcal{R}$-algebra, and has specialisation homomorphisms.
\begin{theorem} [Wang, Section 2.5 \cite{Wang1}] \label{thm:FH_g_spec_hom}
For each $n \in \mathbb{Z}_{\geq 0}$ there is a surjective ring homomorphism $\Phi_n: \FHG \to Z(\mathbb{Z}\Gamma \wr S_n)$ defined by
\[
\Phi_n \left( \sum_{\bs \mu} a_{\bs \mu}(t) K_{\bs\mu} \right) = \sum_{\bs \mu} a_{\bs \mu}(n) X_{\bs \mu},
\]
where $a_{\bs\mu}(t) \in \mathcal{R}$, and $X_{\bs\mu}$ is the sum of all elements of $\Gamma \wr S_n$ of partially-reduced cycle type $\mu$ (or zero, if there are no such elements).
\end{theorem}
\noindent
As we now show, the family of homomorphisms $\Phi_n$ separates elements of $\FHG$. So to prove an identity in $\FHG$, it is enough to pass to $Z(\mathbb{Z}\Gamma \wr S_n)$ by applying $\Phi_n$.
\begin{proposition} \label{specialisation_distinguishing_prop}
Suppose $x,y \in \FHG$ sayisfy $\Phi_n(x) = \Phi_n(y)$ for all sufficiently large positive integers $n$. Then $x=y$.
\end{proposition}
\begin{proof}
Consider an element
\[
z = \sum_{\bs \mu} a_{\bs \mu}(t) K_{\bs\mu}.
\]
Then
\[
\Phi_n \left(z \right) = \sum_{\bs \mu} a_{\bs \mu}(n) X_{\bs \mu},
\]
where $X_{\bs \mu}$ is nonzero provided $n \geq |\bs \mu| + l(\bs \mu(1))$, and such nonzero elements form a basis of $Z(\mathbb{Z}\Gamma \wr S_n)$. This means that $\Phi_n(z)$ determines $a_{\bs \mu}(n)$ for all $n$ sufficiently large, which is a Zariski-dense subset of $\mathbb{Z}$. Hence $\Phi_n(z)$ determines each coefficient $a_{\bs \mu}(t) \in \mathcal{R}$, and so $\Phi_n(z)$ determines $z$ itself. 
\end{proof}
\noindent
We have introduced $\FHG$ as a $\mathcal{R}$-algebra. However, it will turn out that it is better to view is as an algebra over a $\Gamma_*$-weighted version of $\mathcal{R}$, which we now work towards defining.
\begin{definition}
Consider formal power series in variables $x_c$ indexed by $c \in \Gamma_*$. We encode monomials in the $x_c$ using functions $\mathbf{N}: \Gamma_* \to \mathbb{Z}_{\geq 0}$ that record the exponent of each variable:
\[
x^{\mathbf{N}} = \prod_{c \in \Gamma_*} x_c^{\mathbf{N}(c)}.
\]
Additionally, let $|\mathbf{N}| = \sum_c \mathbf{N}(c)$, so that $\deg(x^{\mathbf{N}}) = |\mathbf{N}|$.
\end{definition}
\noindent
Note that the set of such functions $\mathbf{N}$ is $\mathbb{Z}_{\geq 0}^{\Gamma_*}$. 
\begin{definition}
Let $\mathfrak{g} = \mathbb{Q}\Gamma_* = Z(\mathbb{Q}\Gamma)$, viewed as a Lie algebra via the associative multiplication (it is abelian). There is a canonical map $\iota: \mathfrak{g} \to \mathcal{U}(\mathfrak{g})$ from $\mathfrak{g}$ to its universal enveloping algebra, which we extend to formal power series coefficientwise. We define
\[
T: \mathfrak{g}[[x_{c_1}, \ldots, x_{c_l}]] \to \mathcal{U}(\mathfrak{g})[[x_{c_1},  \ldots, x_{c_l}]]
\]
via
\[
T\left(
\sum_{\mathbf{N} \in \mathbb{Z}_{\geq 0}^{\Gamma_*}} u_{\mathbf{N}} x^{\mathbf{N}}
\right) = \sum_{\mathbf{N} \in \mathbb{Z}_{\geq 0}^{\Gamma_*}} \iota(u_{\mathbf{N}}) x^{\mathbf{N}}.
\]
where $u_{\mathbf{N}} \in \mathfrak{g}$.
\end{definition}
\begin{remark}
The map $T$ allows us to distinguish between the associative multiplication in $\mathfrak{g}$, and the multiplication in the universal enveloping algebra. For example, if $u, v \in \mathfrak{g}$, then $uv \in \mathfrak{g}$, while $T(u)T(v)$ is an element of $\mathcal{U}(\mathfrak{g})$ that lies in degree two with respect to the Poincar\'{e}-Birkhoff-Witt (PBW) filtration. It may seem peculiar to consider the universal enveloping algebra of an abelian Lie algebra when the result is simply the symmetric algebra of the Lie algebra. The ring we are about to define, $\RGamma$, is actually a distribution algebra, and so the universal enveloping algebra is actually the natural place to construct it. This is discussed further in the appendix. There is also similar situation when interpolating Grothendieck rings of wreath products (rather than centres of group algebras of wreath products) where $\mathfrak{g}$ is the Grothendieck ring of a tensor category. In that case, $\mathfrak{g}$ is not necessarily abelian and the universal enveloping algebra is the correct construction. The interested reader is directed to \cite{Ryba1} and \cite{Ryba2}.
\end{remark}
\begin{definition} \label{rgamma_constr_defn}
Let $\Omega \in \mathcal{U}(\mathfrak{g})[[x_{c_1}, \ldots, x_{c_l}]]$ be the generating function defined as follows,
\[
\Omega(x_{c_1}, \ldots, x_{c_l}) = \exp\left(T\left(\log\left(1 + \sum_{c \in \Gamma_*} c x_c\right)\right)\right).
\]
Here $\log\left(1 + \sum_{c \in \Gamma_*} c x_c\right)$ is viewed as an element of $\mathfrak{g}[[x_{c_1}, x_{c_2}, \ldots, x_{c_l}]]$ (note that expanding the power series involves the associative algebra structure of $\mathfrak{g}$, not just the Lie algebra structure). Writing $\Omega$ in terms of monomials, we let $B_{\mathbf{N}} \in U(\mathfrak{g})$ be the coefficients:
\[
\Omega(x_{c_1}, \ldots, x_{c_l}) = \sum_{\mathbf{N} \in \mathbb{Z}_{\geq 0}^{\Gamma_*}} B_{\mathbf{N}} x^{\mathbf{N}}.
\]
We let $\RGamma$ be the $\mathbb{Z}$-submodule of $\mathcal{U}(\mathfrak{g})$ spanned by all the $B_{\mathbf{N}}$, and we call $\RGamma$ \emph{the ring of $\Gamma_*$-weighted integer-valued polynomials}.
\end{definition}
\noindent
We will shortly show that $\RGamma$ is indeed a ring, but first, the following example justifies the analogy to integer-valued polynomials.
\begin{example} \label{trivial_rg_example}
Suppose that $\Gamma$ is the trivial group, so that there is only one conjugacy class $c$, and it obeys $c^2 = c$. Then,
\begin{eqnarray*}
\log\left(1 + \sum_{c \in \Gamma_*} c x_c\right) &=& 
\sum_{i\geq 1} \frac{(-1)^{i-1}}{i} c^i x_c^i \\
&=& c \sum_{i\geq 1} \frac{(-1)^{i-1}}{i} x_c^i \\
&=& c \log(1+x_c).
\end{eqnarray*}
Hence $\Omega$ becomes
\begin{eqnarray*}
\Omega(x_{c}) &=& \exp(T(c \log(1+x_c))) \\
&=& \exp(\log(1+x_c) T(c)) \\
&=& (1+x_c)^{T(c)} \\
&=& \sum_{N \geq 0} {T(c) \choose N} x_c^N.
\end{eqnarray*}
So the $B_{\mathbf{N}}$ are simply binomial coefficients (with parameter $T(c)$) and $\RGamma = \mathcal{R}$.
\end{example}

\begin{proposition} \label{theta_product_prop}
We have the following relation:
\[
\Omega(x_{c_1}, \ldots, x_{c_l}) \Omega(y_{c_1}, \ldots, y_{c_l}) = \Omega(z_{c_1}, \ldots, z_{c_l}),
\]
where
\[
z_{c_i} = x_{c_i} + y_{c_i} + \sum_{j,k}A_{j,k}^{i}x_{c_j} y_{c_k}.
\]
\end{proposition}
\begin{proof}
Let $S(x) = 1 + \sum_{c \in \Gamma_*} c x_c$, and similarly define $S(y)$ and $S(z)$. Then
\begin{eqnarray*}
S(x)S(y) &=& \left(1 + \sum_{c \in \Gamma_*} c x_c\right)\left(1 + \sum_{c \in \Gamma_*} c y_c\right) \\
&=& 1 + \sum_{c_j \in \Gamma_*} c_j x_{c_j} + \sum_{c_k \in \Gamma_*} c_k y_{c_k} + \sum_{c_i} A_{j,k}^i c_i x_{c_j}y_{c_k} \\
&=& S(z).
\end{eqnarray*}
Because we are working with commutative algebras, we may compute as follows:
\begin{eqnarray*}
& & \Omega(x_{c_1}, \ldots, x_{c_l}) \Omega(y_{c_1}, \ldots, y_{c_l}) \\
&=& \exp\left(T\left(\log\left(S(x)\right)\right)\right) \exp\left(T\left(\log\left(S(y)\right)\right)\right) \\
&=& \exp\left(T\left(\log\left(S(x)\right)\right) + T\left(\log\left(S(y)\right)\right)\right) \\
&=&\exp\left(T\left(\log\left(S(x)\right)+\log\left(S(y)\right)\right)\right) \\
&=&\exp\left(T\left(\log\left(S(x)S(y)\right)\right)\right) \\
&=&\exp\left(T\left(\log\left(S(z)\right)\right)\right) \\
&=& \Omega(z_{c_1}, \ldots, z_{c_l}).
\end{eqnarray*}
\end{proof}

\begin{lemma} \label{leading_term_lemma}
As an element of $\mathcal{U}(\mathfrak{g})$, $B_{\mathbf{N}}$ lies in PBW filtration degree $|\mathbf{N}|$ and has leading order term
\[
\prod_{c \in \Gamma_*} \frac{T(c)^{\mathbf{N}(c)}}{\mathbf{N}(c)!}.
\]
\end{lemma}
\begin{proof}
The $B_{\mathbf{N}}$ are defined by the equation
\[
\Omega(x_{c_1}, \ldots, x_{c_l}) = \exp\left(T\left(\log\left(1 + \sum_{c \in \Gamma_*} c x_c\right)\right)\right) = \sum_{\mathbf{N} \in \mathbb{Z}_{\geq 0}^{\Gamma_*}} B_{\mathbf{N}} x^{\mathbf{N}}.
\]
The expression
\[
T\left(\log\left(1 + \sum_{c \in \Gamma_*} c x_c\right)\right)
\]
is in PBW degree 1. When the logarithm is expanded as a power series in the $x_c$ variables, the lowest order term is
\[
\sum_{c \in \Gamma_*} T(c) x_c.
\]
In particular, there is no constant term. This means that $\Omega(x_{c_1}, \ldots, x_{c_l})$ is a sum of products of terms whose PBW degree is less than or equal to their degree in the $x_c$ variables. Since $B_{\mathbf{N}}$ is the coefficient of $x^{\mathbf{N}}$, it is contained in PBW filtration degree $|\mathbf{N}|$. Moreover, to compute the leading term of the $B_N$, we neglect all but the lowest degree monomials in $x_c$; to leading order, $\Omega$ is approximated by
\[
\exp\left( \sum_{c \in \Gamma_*} T(c) x_c\right) = 
\prod_{c \in \Gamma_*} \exp\left( T(c) x_c\right).
\]
The leading term of $B_{\mathbf{N}}$ is found by taking the coefficient of $x^{\mathbf{N}}$.
\end{proof}

\begin{proposition} \label{r_gamma_is_an_algebra_prop}
The $B_{\mathbf{N}}$ are a $\mathbb{Z}$-basis of $\RGamma$. Additionally, the multiplication in $\mathcal{U}(\mathfrak{g})$ induces a multiplication on $\RGamma$, making it into a unital commutative ring.
\end{proposition}
\begin{proof}
The $B_{\mathbf{N}}$ span $\RGamma$ by definition. It follows from Lemma \ref{leading_term_lemma} and the PBW theorem that the $B_{\mathbf{N}}$ form a $\mathbb{Q}$-basis of $\mathcal{U}(\mathfrak{g})$, so they are also linearly independent.
\newline \newline \noindent
We note that when $\mathbf{N}$ is zero, $B_{\mathbf{N}}$ is the constant term of $\Omega$, namely the identity in $\mathcal{U}(\mathfrak{g})$, so $\RGamma$ has an identity element. To show that $\RGamma$ is a commutative ring, we note must show that the product of two basis elements is a linear combination of basis elements. This amounts to showing that the coefficient of any monomial $x^{\mathbf{N}} y^{\mathbf{M}}$ in 
\[
\Omega(x_{c_1}, \ldots, x_{c_l}) \Omega(y_{c_1}, \ldots, y_{c_l})
\]
is contained in $\RGamma$. We invoke Proposition \ref{theta_product_prop}, which expresses this product as 
\[
\Omega(z_{c_1}, \ldots, z_{c_l}) = \sum_{\mathbf{K} \in \mathbb{Z}_{\geq 0}^{\Gamma_*}} B_{\mathbf{K}} z^{\mathbf{K}},
\]
where $z_{c_i} = x_{c_i} + y_{c_i} + \sum_{j,k}A_{j, k}^{i}x_{c_j} y_{c_k}$, which is a $\mathbb{Z}$-linear combination of monomials in the $x$ and $y$ variables. Thus the expansion in terms of $x^{\mathbf{N}} y^{\mathbf{M}}$ will have $\mathbb{Z}$-linear combinations of the $B_{\mathbf{K}}$ as coefficients.
\end{proof}

\begin{proposition} \label{rgamma_free_over_r_prop}
The set of $B_{\mathbf{N}}$ in $\RGamma$ indexed by $\mathbf{N}$ with $\mathbf{N}(c) = 0$ for $c \neq 1$ forms a subring of $\RGamma$ isomorphic to $\mathcal{R}$. As a module over $\mathcal{R}$, $\RGamma$ is free with basis $B_{\mathbf{M}}$ indexed by $\mathbf{M}$ with $\mathbf{M}(1) = 0$.
\end{proposition}

\begin{proof}
The $B_{\mathbf{N}}$ with $\mathbf{N}(c) = 0$ for $c \neq 1$ can be found by considering the generating function $\Omega$ and setting all $x_c$ to zero other than $x_1$. Then the calculation in Example \ref{trivial_rg_example} shows $B_{\mathbf{N}} = {T(1) \choose \mathbf{N}(1)}$, and binomial coefficients are a $\mathbb{Z}$-basis of $\mathcal{R}$.
\newline \newline \noindent
To show that $\RGamma$ is free over $\mathcal{R}$ with the stated basis, we pass to the associated graded algebra with respect to the PBW filtration. Lemma \ref{leading_term_lemma} showed the leading order term of $B_{\mathbf{N}}$ is 
\[
\prod_{c \in \Gamma_*} \frac{T(c)^{\mathbf{N}(c)}}{\mathbf{N}(c)!} = 
\frac{T(1)^{\mathbf{N}(1)}}{\mathbf{N}(1)!}  \prod_{c \neq 1} \frac{T(c)^{\mathbf{N}(c)}}{\mathbf{N}(c)!},
\]
which we recognise as the leading term of ${T(1) \choose \mathbf{N}(1)} \in \mathcal{R}$ multiplied by the leading term of $B_{\mathbf{M}}$, where $\mathbf{M}(c) = \mathbf{N}(c)$ for $c \neq 1$ and $\mathbf{M}(1) = 0$. This shows that $\RGamma$ is spanned by the $B_{\mathbf{M}}$ over $\mathcal{R}$, and also shows the $\mathcal{R}$-linear independence of the $B_{\mathbf{M}}$.
\end{proof}

\begin{definition}
For any $\mathbf{N} \in \mathbb{Z}_{\geq 0}^{\Gamma_*}$ with $|\mathbf{N}| \leq n$, we define $b_{\mathbf{N}} \in \mathbb{Q}\Gamma \wr S_n$ as follows
\[
b_{\mathbf{N}} = \frac{1}{(n - |\mathbf{N}|)! \prod_{c \in \Gamma_*} \mathbf{N}(c)!}
\sum_{\sigma \in S_n} \sigma \left( 1^{\otimes (n - |\mathbf{N}|)} \otimes \bigotimes_{c \in \Gamma_*} c^{\otimes \mathbf{N}(c)} \right) \sigma^{-1}.
\]
The parenthesised tensor product is an element of $(\mathbb{Q} \Gamma)^{\otimes n} = \mathbb{Q}\Gamma^n \subseteq \mathbb{Q}\Gamma \wr S_n$ (because of the averaging over $S_n$, the choice of order of the tensor factors does not matter).
\end{definition}
\begin{lemma}
The element $b_{\mathbf{N}}$ is actually contained in $Z(\mathbb{Z}\Gamma \wr S_n)$.
\end{lemma}
\begin{proof}
Since $b_{\mathbf{N}}$ is defined by averaging over $S_n$, it commutes with $S_n \subseteq \Gamma \wr S_n$. As the element being averaged is a tensor product of central elements of $\mathbb{Z}\Gamma$, $b_{\mathbf{N}}$ will commute with $\Gamma^n \subseteq \Gamma \wr S_n$. This shows that $b_{\mathbf{N}}$ is central. To see that $b_{\mathbf{N}}$ has integer coefficients, note that the conjugation action of the subgroup 
\[
S_{n - |\mathbf{N}|} \times \prod_{c \in \Gamma_*} S_{\mathbf{N}(c)} \subseteq S_n
\]
is trivial. Any element of a fixed coset of this subgroup acts the same way, and the size of this coset cancels out the denominator.
\end{proof}

\begin{lemma} \label{specialised_gen_fun_lemma}
We have the following expression for the generating function for the elements $b_{\mathbf{N}}$:
\[
\sum_{|\mathbf{N}| \leq n} b_{\mathbf{N}} x^{\mathbf{N}} = (1 + \sum_c c x_c)^{\otimes n}.
\]
\end{lemma}

\begin{proof}
We observe that $b_{\mathbf{N}}$ is the sum of pure tensors with $\mathbf{N}(c)$ tensor factors set equal to $c$ and all remaining tensor factors equal to $1$. Here we distinguish tensor factors assigned $c=1$ from unassigned tensor factors which also take the value 1. But the coefficient of $x^{\mathbf{N}}$ on the right hand side will be the sum of all pure tensors with $\mathbf{N}(c)$ factors equal to $c$. 
\end{proof}

\begin{theorem} \label{r_gamma_hom_thm}
There is a homomorphism $ev_n^{\mathcal{R}}: \RGamma \to Z(\mathbb{Z} \Gamma \wr S_n)$ defined by $ev_n^{\mathcal{R}}(B_{\mathbf{N}}) = b_{\mathbf{N}}$ if $|\mathbf{N}| \leq n$, and $ev_n^{\mathcal{R}}(b_{\mathbf{N}}) = 0$ otherwise.
\end{theorem}

\begin{proof}
We apply $ev_n^{\mathcal{R}}$ to $\Omega$ coefficientwise.
Using the generating function from Lemma \ref{specialised_gen_fun_lemma},
\begin{eqnarray*}
ev_n^{\mathcal{R}}(\Omega(x_{c_1}, \ldots, x_{c_l}))ev_n^{\mathcal{R}}( \Omega(y_{c_1}, \ldots, y_{c_l})) &=& 
(1 + \sum_c cx_c)^{\otimes n} (1 + \sum_c c y_c)^{\otimes n} \\
&=& ((1 + \sum_c cx_c)(1 + \sum_c c y_c))^{\otimes n} \\
&=& (1 + \sum_c c z_c)^{\otimes n} \\
&=& ev_n^{\mathcal{R}}(\Omega(z_{c_1}, \ldots, z_{c_l})) \\
&=& ev_n^{\mathcal{R}}(\Omega(x_{c_1}, \ldots, x_{c_l})\Omega(y_{c_1}, \ldots, y_{c_l}))
\end{eqnarray*}
where $z_{c_i} = x_{c_i} + y_{c_i} + \sum_{j,k}A_{j,k}^{i} x_{c_j} y_{c_k}$ as in Proposition \ref{r_gamma_is_an_algebra_prop}. Considering the coefficients of monomials $x^{\mathbf{N}}y^{\mathbf{M}}$ shows $ev_n^{\mathcal{R}}(B_{\mathbf{N}})ev_n^{\mathcal{R}}(B_{\mathbf{M}}) = ev_n^{\mathcal{R}}(B_{\mathbf{N}}B_{\mathbf{M}})$, i.e. that $ev_n^{\mathcal{R}}$ is a homomorphism.
\end{proof}

\noindent
The elements $B_{\mathbf{N}}$ can be thought of as interpolating the $b_{\mathbf{N}}$. To make this precise, we show that $ev_n^{\mathcal{R}}$ factors through $\FHG$.

\begin{corollary} \label{coeff_compatibility_cor}
We obtain a canonical homomorphism $\Psi: \RGamma \to \FHG$ via
\[
\Psi( B_{\mathbf{N}}) = {t - |\mathbf{N}| + \mathbf{N}(1) \choose \mathbf{N}(1)} K_{\bs \mu},
\]
where $t$ is the polynomial variable of $\mathcal{R}$, $\bs \mu(c) = (1^{\mathbf{N}(c)})$ for $c \neq 1$, and $\bs \mu(1)$ is the empty partition. This homomorphism obeys $\Phi_n \circ \Psi = ev_n^{\mathcal{R}}$.
\end{corollary}

\begin{proof}
We first show that $\Phi_n \circ \Psi = ev_n^{\mathcal{R}}$. This amounts to checking that
\[
b_{\mathbf{N}} = {n - |\mathbf{N}| + \mathbf{N}(1) \choose \mathbf{N}(1)} X_{\bs \mu}.
\]
Each pure tensor in $b_{\mathbf{N}}$ is an element of generalised cycle type $\bs \mu$. However, such a pure tensor has $\mathbf{N}(1)$ assigned factors equal to $1$ and $n - |\mathbf{N}|$ unassigned factors equal to $1$. So each such pure tensor arises in ${n - |\mathbf{N}| + \mathbf{N}(1) \choose \mathbf{N}(1)}$ ways according to how the $\mathbf{N}(1)$ assigned factors are chosen. This accounts for the multiplicative factor. To check that $\Psi$ respects multiplication, let $x,y \in \RGamma$. Because $ev_n^{\mathcal{R}}$ and $\Phi_n$ are homomorphisms,
\begin{eqnarray*}
\Phi_n(\Psi(x)\Psi(y)) &=& \Phi_n(\Psi(x))\Phi_n(\Psi(y)) \\
&=& ev_n^{\mathcal{R}}(x) ev_n^{\mathcal{R}}(y) \\
&=& ev_n^{\mathcal{R}}(xy) \\
&=& \Phi_n(\Psi(xy)).
\end{eqnarray*}
We now use Proposition \ref{specialisation_distinguishing_prop} to conclude that $\Psi(x)\Psi(y) = \Psi(xy)$.
\end{proof}
\noindent
The upshot is that we may now view $\FHG$ as a $\RGamma$-algebra, not merely a $\mathcal{R}$-algebra. But the fact that our ground ring now incorporates all conjugacy classes of $\Gamma$ means that rather than considering partially reduced cycle types, we will have to reduce all parts of a multipartition.
\begin{definition}
If an element of $\Gamma \wr S_n$ has cycle type $\bs \mu$, we say it has \emph{fully-reduced cycle type} $\bs \nu$, where for each $c \in \Gamma_*$, $\bs \nu(c)$ is obtained from $\bs \mu(c)$ by subtracting $1$ from each part, and ignoring any resulting parts of size zero. We let $\bm\hat{\bs \mu}$ be the multipartition obtained from $\bs \mu$ by adding 1 to each nonzero part of $\bs \mu(c)$ for each $c \neq 1$, and leaving $\bs \mu(1)$ unchanged.
\end{definition}
\noindent
For an element of $\Gamma \wr S_n$, passing from the cycle type to the fully-reduced cycle type is generally not reversible, even if $n$ is known. While we may recover the number of $1$-cycles, their types may be arbitrary conjugacy classes. The notation has been set up so that $X_{\bm\hat{\bs \mu}}$ is the sum of all elements of $\Gamma \wr S_n$ with all 1-cycles labelled by the trivial conjugacy class and having fully-reduced cycle type $\bs \mu$.
\begin{theorem}
As $\RGamma$-module, $\FHG$ is free with basis $K_{\bm \hat{\bs \mu}}$ where $\mu \in \mathcal{P}(\Gamma_*)$.
\end{theorem}
\begin{proof}
By Proposition \ref{rgamma_free_over_r_prop}, $\RGamma$ is free over $\mathcal{R}$ with basis $B_{\mathbf{M}}$ indexed by $\mathbf{M}$ with $\mathbf{M}(1) = 0$. So it is enough to show that $K_{\bm\hat{\bs\mu}}B_{\mathbf{M}}$ is a $\mathcal{R}$-basis of $\FHG$. To accomplish this, we introduce a filtration of $\FHG$ by $\mathcal{R}$-modules by placing $K_{\bs \mu}$ in filtration degree $|\bs \mu| + l(\bs\mu(1))$. We now show that this respects multiplication, by passing to $Z(\mathbb{Z} \Gamma \wr S_n)$.
\newline \newline \noindent
Let us say that $(\mathbf{g}, \sigma) \in \Gamma \wr S_n$ \emph{affects} $i \in \{1,\ldots, n\}$ if either $\sigma(i) \neq i$, or $\mathbf{g}_i \neq 1$. The number of elements of $\{1,\ldots,n\}$ that are affected by an element of partially-reduced cycle type $\bs \mu$ is precisely $|\bs \mu| + l(\bs\mu(1))$. Consider two elements $(\mathbf{g}, \sigma), (\mathbf{h}, \rho) \in \Gamma \wr S_n$, and suppose that their product $(\mathbf{g} \sigma(\mathbf{h}), \sigma \rho)$ affects $i$. Then either $\rho(\sigma(i)) \neq i$, or $\mathbf{g}_i\sigma(\mathbf{h})_i \neq 1$. If $\rho(\sigma(i)) \neq i$, then either $\sigma(i) \neq i$, or $\rho(i) \neq i$, which implies that one of $(\mathbf{g}, \sigma), (\mathbf{h}, \rho)$ affects $i$. If instead $\mathbf{g}_i\sigma(\mathbf{h})_i \neq 1$, either $\mathbf{g}_i \neq 1$ or $\sigma(\mathbf{h})_i \neq 1$. In the first of these cases, $(\mathbf{g}, \sigma)$ affects $i$. In the second case, either $\sigma(i) \neq i$, or $\mathbf{h}_i \neq 1$. Hence we conclude that if $(\mathbf{g} \sigma(\mathbf{h}), \sigma \rho)$ affects $i$, then one of $(\mathbf{g}, \sigma)$ and $(\mathbf{h}, \rho)$ affects $i$. This implies that the set of elements in $\{1,\ldots,n\}$ affected by the product of two group elements is a subset of the union of the elements affected by each factor. This proves that we have a filtration, and shows that only products of elements affecting disjoint subsets of $\{1,\ldots,n\}$ contribute to the leading term in the associated graded ring.
\newline\newline\noindent
To compute $K_{\bm\hat{\bs\mu}}B_{\mathbf{M}}$, we again pass to $\mathbb{Z}\Gamma \wr S_n$. This gives $X_{\bm\hat{\bs\mu}}b_{\mathbf{M}}$. The leading order term arises from $X_{\bm\hat{\bs\mu}}$ and $b_{\mathbf{M}}$ affecting disjoint subsets of $\{1,\ldots, n\}$. In that case, we obtain $X_{\bs \nu}$, where $\bs \nu(c) = \bm \hat{\bs \mu}(c) \cup (1^{\mathbf{M}(c)})$ for $c \in \Gamma_*$. Any multipartition $\bs \nu$ arises from exactly one pair $(\bm\hat{\bs\mu}, \mathbf{M})$ in this way. In particular $\bs \mu$ is the fully-reduced cycle type of corresponding to the partially-reduced cycle type $\bs \nu$, while $\mathbf{M}$ records how many parts of $\bs \nu$ of size $1$ have a given label in $\Gamma_*\backslash\{1\}$. We conclude that $K_{\bm\hat{\bs\mu}}B_{\mathbf{M}}$ equals $K_{\bs \nu}$ plus lower order terms. In particular the $K_{\bm\hat{\bs\mu}}B_{\mathbf{M}}$ form an $\mathcal{R}$-basis of $\FHG$ and the theorem follows.
\end{proof}

\section{Isomorphism with \texorpdfstring{$\Gamma_*$}{Gamma*}-Weighted Symmetric Functions}
\noindent
This section is dedicated to proving that $\FHG$ is isomorphic to $\RGamma \otimes \LambdaG$ via Jucys-Murphy evaluation, analogously to Theorem \ref{thm:FH_lambda_iso}.
\newline \newline \noindent
We introduce two filtrations. Let us say that $\sigma \in S_n$ \emph{moves} an element $m \in \{1,\ldots, n\}$ if $\sigma(m) \neq m$ (i.e. $m$ is not a fixed point of $\sigma$). The first filtration is obtained by placing an element $(\mathbf{g}, \sigma) \in \Gamma \wr S_n$ in degree $i$, where $i$ is the number of elements of $\{1, \ldots, n\}$ which are moved by $\sigma$. 
\begin{definition}
The \emph{moving filtration}, $\mathcal{F}_{mov}^i$, is the family of $\RGamma$-submodules of $\FHG$ spanned by $K_{\bm\hat{\bs \mu}}$ with $|\bs \mu| + l(\bs \mu) \leq i$.
\end{definition}

\begin{proposition} \label{moving_filt_prop}
The moving filtration is an algebra filtration. In the associated graded algebra, we have the following equation:
\[
K_{\bm\hat{\bs \mu}}K_{\bm\hat{\bs \nu}} = \left( \prod_{c \in \Gamma_*} \prod_{i \geq 1} \frac{m_i(\bs\lambda(c))!}{m_i(\bs\mu(c))!m_i(\bs\nu(c))!}\right) K_{\bm\hat{\bs \lambda}},
\]
where $m_i(\bs\lambda(c)) = m_i(\bs\mu(c)) + m_i(\bs\nu(c))$.
\end{proposition}

\begin{proof}
If $\sigma \in S_n$ has reduced cycle type $\mu$, the number of elements moved by $\sigma$ is $|\mu| + l(\mu)$. Similarly if $(\mathbf{g}, \sigma) \in \Gamma \wr S_n$ has fully-reduced cycle type $\bs\mu$, then the number of elements moved by $\sigma$ is $|\bs \mu| + l(\bs \mu)$. Note that this corresponds to the condition for membership in $\mathcal{F}_{mov}^i$, where $i = |\bs \mu| + l(\bs \mu)$. 
\newline \newline \noindent
Fix $\bs \mu$ and $\bs \nu$ with $|\bs \mu| + l(\bs \mu) = i$ and $|\bs \nu| + l(\bs \nu) = j$. Consider $K_{\bm\hat{\bs \mu}}K_{\bm\hat{\bs \nu}}$. To compute this product, we pass to $Z(\mathbb{Z}\Gamma \wr S_n)$ for large $n$. A term in $X_{\bm\hat{\bs \mu}}$ is of the form $(\mathbf{g}, \sigma)$, where $\sigma$ moves $i$ elements. Similarly a term in $X_{\bm\hat{\bs \nu}}$ is of the form $(\mathbf{h}, \rho)$, where $\rho$ moves $j$ elements. Now, $(\mathbf{g}, \sigma) \cdot (\mathbf{h}, \rho) = (\mathbf{g}\sigma(\mathbf{h}), \sigma\rho)$. The permutation $\sigma \rho$ is obtained by moving $j$ elements, and then moving $i$ elements. This means that $\sigma \rho$ moves at most $i+j$ elements, with equality if and only if the elements moved by $\sigma$ and $\rho$ are disjoint. This proves the assertion that $\mathcal{F}_{mov}^i$ defines a filtration.
\newline \newline \noindent
If $\sigma$ and $\rho$ move disjoint elements, then $\sigma(\mathbf{h}) = \mathbf{h}$ because the $m$-th factor in $\mathbf{h}$ is equal to $1$ if $m$ is not moved by $\rho$ (so in particular if $m$ is moved by $\sigma$). Hence the cycles (of size larger than 1) of $(\mathbf{g}, \sigma) \cdot (\mathbf{h}, \rho)$ are the constituent cycles of $(\mathbf{g}, \sigma)$ and $(\mathbf{h}, \rho)$. This shows that in the associated graded algebra, $K_{\bm\hat{\bs \mu}}K_{\bm\hat{\bs \nu}}$ is equal to a multiple of $K_{\bm\hat{\bs \lambda}}$. The multiple is equal to the number of ways to split the cycles of $\bs \lambda$ among $\bs \mu$ and $\bs \nu$.
\end{proof}

\noindent
The second filtration is obtained by placing an element $(\mathbf{g}, \sigma) \in \Gamma \wr S_n$ in degree $i$, where $i$ is the smallest number of factors required to express $\sigma$ as a product of transpositions. 
\begin{definition}
The \emph{transposition filtration}, $\mathcal{F}_{tsp}^i$, is the family of $\RGamma$-submodules of $\FHG$ spanned by $K_{\bm\hat{\bs \mu}}$ with $|\bs \mu| \leq i$.
\end{definition}

\begin{proposition} \label{length_filtration_prop}
The transposition filtration is an algebra filtration. \end{proposition}

\begin{proof}
Similarly to the other filtration, we pass to $Z(\mathbb{Z}\Gamma \wr S_n)$ for $n$ sufficiently large. Fix $\bs \mu$ and $\bs \nu$ with $|\bs \mu| = i$ and $|\bs \nu| = j$. A term in $X_{\bm\hat{\bs \mu}}$ is of the form $(\mathbf{g}, \sigma)$, where $\sigma$ may be written as the product of $i$ transpositions. Similarly a term in $X_{\bm\hat{\bs \nu}}$ is of the form $(\mathbf{h}, \rho)$, where $\rho$ is a product of $j$ transpositions. Then, $(\mathbf{g}, \sigma) \cdot (\mathbf{h}, \rho) = (\mathbf{g}\sigma(\mathbf{h}), \sigma\rho)$. The permutation $\sigma \rho$ may be written as a product of $i+j$ transpositions (although possibly fewer). We conclude that $\mathcal{F}_{tsp}^i$ defines a filtration.
\end{proof}

\begin{remark}
Using Lemma 3.2 of \cite{FarahatHigman} and Proposition 2.9 of \cite{Wang1}, one can show that the structure constants of the $K_{\bm\hat{\bs\mu}}$ elements in the associated graded algebra are integers rather than arbitrary elements of $\RGamma$. We will not need this fact.
\end{remark}

\subsection{Jucys-Murphy Elements}
Analogously to the JM elements associated to symmetric groups, Pushkarev \cite{Pushkarev} and Wang \cite{Wang2} independently introduced the following elements of $\mathbb{Z} \Gamma \wr S_n$:
\[
L_j =
\sum_{i<j} \sum_{g \in \Gamma} g^{(i)} (g^{-1})^{(j)} (i, j).
\]
In the case where $\Gamma$ is the trivial group, we recover the JM elements for symmetric groups. However, these elements are not sufficient for our purposes, and we must introduce further generalisations.

\begin{definition}
For $c \in \Gamma_*$ and $j \leq n$, the \emph{wreath-product Jucys-Murphy elements} are
\[
L_j(c) = \sum_{i < j} \sum_{\substack{g_1, g_2 \in \Gamma \\ g_2g_1 \in c}} g_1^{(i)} g_2^{(j)} (i, j),
\]
which are elements of $\mathbb{Z} \Gamma \wr S_n$. We extend this notation by linearity in $c$: if $c = \sum_i  n_i c_i$ for some $n_i \in \mathbb{Z}$, we let
\[
L_j(c) = \sum_{i=1}^l n_i L_j(c_i).
\]
\end{definition}
\noindent
In the case where $c = 1$, the condition $g_2g_1 \in c$ becomes $g_2 = g_1^{-1}$, so $L_j(1)$ recovers the definition of Pushkarev and Wang.
\begin{proposition}
Let $c \in \mathbb{Z}\Gamma_*$. We have
\[
L_j(c) = c^{(j)} L_j(1) = L_j(1) c^{(j)}.
\]
\end{proposition}
\begin{proof}
By lienarity, it suffices to consider $c \in \Gamma_*$. To show the first equality, it suffices that for $i<j$,
\[
c^{(j)} \sum_{\substack{g_1, g_2 \in \Gamma \\ g_2g_1 = 1}} g_1^{(i)} g_2^{(j)} = \sum_{\substack{h_1, h_2 \in \Gamma \\ g_2g_1 \in c}} h_1^{(i)} h_2^{(j)},
\]
as we may right-multiply by $(i,j)$ and sum over $i$ less than $j$. We may set $g_2 = g_1^{-1}$ and use the fact that $c^{(j)}$ commutes with $g_1^{(i)}$ and $g_2^{(j)}$. The left hand side becomes
\[
\sum_{g_1 \in \Gamma} g_1^{(i)} (g_1^{-1})^{(j)} c_j^{(j)} = \sum_{g_1 \in \Gamma, k \in c} g_1^{(i)} (g_1^{-1} k)^{(j)},
\]
which agrees with the right hand side upon identifying $h_1 = g_1$ and $h_2 = g_1^{-1}k$. The second equality is similar, except that moving $c^{(j)}$ past the transposition $(i, j)$ turns it into $c^{(i)}$, which means that the manipulations take place in the $i$-index, rather than the $j$-index.
\end{proof}

\begin{definition}
For $c \in \Gamma_*$, let
\[
M_c^{(i,i+1)} = \sum_{\substack{g,h \in \Gamma \\ hg \in c}} g^{(i)}h^{(i+1)}.
\]
\end{definition}

\noindent
The following lemmas have routine proofs, so we omit them.
\begin{lemma}
We have $M_c^{(i,i+1)} = c^{(i)} M_1^{(i,i+1)}  = c^{(i+1)} M_1^{(i,i+1)}$, and $M_c^{(i,i+1)}$ commutes with the transposition $s_i = (i, i+1)$.
\end{lemma}

\begin{lemma}[Pushkarev, Proposition 5 \cite{Pushkarev}]
We have that the $L_j(1)$ ($j=1,\ldots,n$) commute with each other. The $L_j(1)$ commute with with $\Gamma^n$. Furthermore
\[
s_i L_i(1) s_i + M_1^{(i,i+1)} s_i = L_{i+1}(1).
\]
\end{lemma}
\noindent
The following lemma is easily proven by induction.
\begin{lemma} \label{iterative_swap_lemma}
We have the two identities
\[
s_i L_i(1)^r = L_{i+1}(1)^r s_i - \sum_{p=0}^{r-1} L_{i+1}(1)^p M_1^{(i,i+1)} L_i(1)^{r-1-p}
\]
and
\[
s_i L_{i+1}(1)^r = L_{i}(1)^r s_i + \sum_{p=0}^{r-1} L_{i}(1)^p M_1^{(i,i+1)} L_{i+1}(1)^{r-1-p}.
\]
\end{lemma}

\begin{proposition} \label{special_commuting_proposition}
Suppose $c_a, c_b \in \Gamma_*$. The elements $L_i(1)^r c_a^{(i)}c_b^{(i+1)} + L_{i+1}(1)^r c_b^{(i)} c_a^{(i+1)}$ and $L_i(1)^r L_{i+1}(1)^r$ both commute with $s_i$.
\end{proposition}

\begin{proof}
For the first part, we add the first equation in Lemma \ref{iterative_swap_lemma} right-multiplied by $c_a^{(i)}c_b^{(i+1)}$ to the second equation right-multiplied by $ c_b^{(i)}c_a^{(i+1)}$. Because $L_i(1)$ and $L_{i+1}(1)$ commute with $\Gamma^n$, they commute with $M_1^{(i,i+1)}$. So since $M_1^{(i,i+1)}  c_a^{(i)}c_b^{(i+1)} =  M_1^{(i.i+1)}c_b^{(i)}c_a^{(i+1)}$, the sums over the variable $p$ cancel out, leaving the required identity. For the second part it suffices to consider the case $r=1$ as the general case is recovered by raising $L_i(1)L_{i+1}(1)$ to a suitable power. Then we have
\begin{eqnarray*}
s_i L_i(1) L_{i+1}(1) &=& (L_{i+1}(1)s_i - M_1^{(i,i+1)}) L_{i+1}(1) \\
&=& L_{i+1}(1) (L_i(1)s_i + M_1^{(i,i+1)}) - M_1^{(i,i+1)} L_{i+1}(1) \\
&=& L_{i+1}(1)L_i(1)s_i.
\end{eqnarray*}
Since $L_i(1)$ and $L_{i+1}(1)$ commute, $s_i L_i(1)L_{i+1}(1) = L_{i}(1)L_{i+1}(1)s_i$ and we are done.
\end{proof}

\begin{proposition}
The wreath-product JM elements commute with each other: $L_{j_1}(c_1) L_{j_2}(c_2)=L_{j_2}(c_2) L_{j_1}(c_1)$ for any $c_1, c_2 \in \Gamma_*$ and $j_1, j_2 \in \{1,\ldots, n\}$.
\end{proposition}
\begin{proof}
We have $L_{j_1}(c_1) = L_{j_1}(1) c_1^{(j_1)}$ and $L_{j_2}(c_2) = L_{j_2}(1) c_2^{(j_2)}$ and $L_{j_1}(1), c_1^{(j_1)}, L_{j_2}(1), c_2^{(j_2)}$ commute pairwise.
\end{proof}

\noindent
Because the wreath-product JM elements commute, it makes sense to evaluate a polynomial in them.
\begin{proposition}
There is a ring homomorphism
\[
ev_n^Q: Q^{\otimes n} \to \mathbb{Z} \Gamma \wr S_n,
\]
defined by 
\[
ev_n^Q(x_d^r(c)) = L_d(1)^r c^{(d)}.
\]
Moreover, the image of $\Lambda_n(\Gamma_*)$ (the $S_n$-invariants of $Q^{\otimes n}$) is contained in the centre of $\mathbb{Z} \Gamma \wr S_n$.
\end{proposition}
\begin{proof}
Since the wreath-product JM elements commute pairwise, it suffices to show that $ev_n^Q$ is well-defined on each tensor factor of $Q^{\otimes n}$. The $d$-th factor of $Q$ has a basis $x_d^r(c)$ for $r \geq 0$ and $c \in \Gamma_*$ (where we require $c=1$ if $r=0$). This basis obeys the multiplication rule
\[
x_d^r(c_i) x_d^s(c_j) = x_d^{r+s}(c_ic_j) = \sum_k A_{i,j}^k x_d^{r+s}(c_k).
\]
Correspondingly,
\begin{eqnarray*}
ev_n^Q(x_d^r(c_i))ev_n^Q(x_d^s(c_j)) &=& L_d(1)^r c_i^{(d)} L_d(1)^s c_j^{(d)} \\
&=& L_d(1)^{r+s} (c_ic_j)^{(d)} \\
&=& L_d(1)^{r+s} \sum_k A_{i,j}^k c_k^{(d)} \\
&=& \sum_k A_{i,j}^k ev_n^Q(x_d^{r+s}(c_k)).
\end{eqnarray*}
We conclude that $ev_n^Q$ is well defined. 
\newline \newline \noindent
Now suppose that $P(x_1, \ldots, x_n) \in (Q^{\otimes n})^{S_n}$. To show $ev_n^Q(P)$ is central, it suffices to show that it commutes with a generating set of $\Gamma \wr S_n$. In particular, we may take $\Gamma^n$ together with the adjacent transpositions $s_i = (i,i+1)$ for $i \in \{1,\ldots, n-1\}$. Since $L_j(1)$ commutes with $\Gamma^n$, the same is true of $L_j(c) = c^{(j)}L_j(1)$, from which it follows that the entire image of $ev_n^Q$ commutes with $\Gamma^n$. Now we show that $ev_n^Q(P)$ commutes with $s_i$. Since $P \in (Q^{\otimes n})^{S_n}$, $P$ is in particular an element of
\[
Q^{\otimes (i-1)} \otimes (Q \otimes Q)^{S_2} \otimes Q^{\otimes (n-i-1)},
\]
and $(Q \otimes Q)^{S_2}$ is spanned by elements of the form $x_i^r(c) x_{i+1}^r(c)$ and $x_i^r(c_a) x_{i+1}^s(c_b) + x_i^s(c_b) x_{i+1}^r(c_a)$. So it is enough that applying $ev_n$ to these elements gives something that commutes with $s_i$. In the first case,
\[
ev_n^Q(x_i^r(c) x_{i+1}^r(c)) = L_i(1)^r c^{(i)} L_{i+1}(1)^r c^{(i+1)},
\]
which commutes with $s_i$ because $L_i(1) L_{i+1}(1)$ and $c^{(i)}c^{(i+1)}$ do. For the second case, without loss of generality, we assume $s \leq r$. Now,
\begin{eqnarray*}
ev_n^Q(x_i^r(c_a) x_{i+1}^s(c_b) + x_i^s(c_b) x_{i+1}^r(c_a)) &=& L_i(1)^r c_a^{(i)} L_{i+1}(1)^s c_b^{(i+1)} + L_i(1)^r c_b^{(i)} L_{i+1}(1)^s c_a^{(i+1)}\\
&=& L_i(1)^s L_{i+1}(1)^s \left( L_i(1)^{r-s}c_a^{(i)} c_b^{(i+1)} + L_{i+1}(1)^{r-s}c_b^{(i)} c_a^{(i+1)} \right)
\end{eqnarray*}
is a product of elements that commute with $s_i$ from Proposition \ref{special_commuting_proposition}.
\end{proof}

\begin{definition} \label{ev_n_formal_def}
Let $ev_n^\Lambda: \LambdaG \to Z(\mathbb{Z}\Gamma \wr S_n)$ be the composition of the canonical map $\LambdaG \to \Lambda_n(\Gamma_*) \subseteq Q^{\otimes n}$ with $ev_n^Q: Q^{\otimes n} \to Z(\mathbb{Z}\Gamma \wr S_n)$.
\end{definition}
\noindent
Finally, we combine our evaluation maps into a single map.
\begin{definition}
For each $n\in \mathbb{Z}_{\geq 0}$ there is a ring homomorphism
\[
ev_n: \RGamma \otimes \LambdaG \to Z(\mathbb{Z}\Gamma \wr S_n)
\]
defined as 
\[
ev_n( a \otimes f) = ev_n^{\mathcal{R}}(a) ev_n^{\Lambda}(f),
\]
using the maps $ev_n^{\mathcal{R}}$ and $ev_n^{\Lambda}$ from Theorem \ref{r_gamma_hom_thm} and Definition \ref{ev_n_formal_def}, respectively.
\end{definition}

\begin{theorem} \label{wreath_map_exists}
There is a unique homomorphism $\Psi: \RGamma \otimes \LambdaG \to \FHG$ extending the map in Corollary \ref{coeff_compatibility_cor} such that the following diagram commutes.
\begin{equation*}
\begin{tikzcd}
\RGamma \otimes_{\mathbb{Z}} \LambdaG    \arrow[r, "\Psi"]    \arrow[dr, "ev_n"]    &    \FHG    \arrow[d, "\Phi_n"]\\
&    Z(\mathbb{Z}S_n)
\end{tikzcd}
\end{equation*}
\end{theorem}
\begin{proof}
Corollary \ref{coeff_compatibility_cor} constructs $\Psi$ on $\RGamma$, so it remains to construct $\Psi$ on $\LambdaG$. To do this, we pass to complex coefficients, and then show that the map is defined integrally. Proposition \ref{complex_bas_change_prop} shows
\[
\mathbb{C} \otimes \LambdaG = \mathbb{C} \otimes \Lambda^{\otimes |\Gamma^*|}
\]
by changing basis in $\mathbb{C}\Gamma_*$ from conjugacy-class sums to primitive central idempotents $e_\chi$. To compute $ev_n$ on the tensor factor of $\Lambda$ corresponding to $\chi \in \Gamma^*$, we must evaluate symmetric functions in the variables
\[
L_j(e_\chi) = L_j(1) e_\chi^{(j)}.
\]
Since $\Lambda$ is freely generated by the elementary symmetric functions $e_r$, it suffices to consider the evaluation of $e_r$ at the $L_j(e_\chi)$. Momentarily ignoring the $\Gamma^n$ parts in $L_j(e_\chi)$, we see that the transpositions involved in $L_j(e_\chi)$ are the same as in the JM elements $L_j$ for the symmetric group. So the $S_n$ parts are the same as in Theorem \ref{thm:elem_JM}. Lemma \ref{cycle_label_product} tells us that the labels of intersecting cycles multiply. But each label is the idempotent $e_\chi$. So we get a sum of elements whose reduced cycle type has size $r$, and all of whose cycles are labelled by $e_\chi$ (which is a linear combination of the usual labels in $\Gamma_*$). We could write this
\[
e_r(L_1(e_\chi), \ldots, L_n(e_\chi)) = \sum_{|\mu| = r} X_{\mu_\chi}^{irr},
\]
where the meaning of $X_{\mu_\chi}^{irr}$ is analogous to $m_{\bs\lambda_\chi}$ from Proposition \ref{complex_bas_change_prop}. Changing bases of $\mathbb{C}\Gamma_*$ back from $e_\chi$ to conjugacy-class sums changes the labels back to conjugacy classes. This is a linear operation that is independent of $n$. Hence we obtain some element
\[
e_r(L_1(e_\chi), \ldots, L_n(e_\chi)) = \sum_{|\bs\nu| = r} R_{r, \bs\nu}^\chi X_{\bs\hat{\bs \nu}},
\]
where $R_{r, \bs\nu}^\chi$ is independent of $n$, and the multipartition $\bs\nu$ is indexed by $\Gamma_*$ (the sizes of the cycles of the elements are unchanged by changing the type so we still sum over multipartitions of size $r$). Thus for $e_r \in \Lambda(\chi)$, we may take
\[
\Psi(e_r) = \sum_{|\bs\nu| = r} R_{r, \bs\nu}^\chi K_{\bm\hat{\bs \nu}},
\]
which guarantees that $ev_n = \Phi_n \circ \Psi$.
\newline \newline \noindent
To see that the image of $\Psi$ on $\LambdaG$ is contained in $\FHG$ rather than $\mathbb{C}\otimes \FHG$, note that we can tell whether an element $f \in \mathbb{C} \otimes \FHG$ is contained in $\FHG$ by applying $\Phi_n$ for sufficiently large $n$ and checking whether the coefficients of the conjugacy class sums $X_{\bs\mu}$ are integers. This would imply that the coefficient of $K_{\bs\mu}$ in $f$, which is an element of $\mathbb{C} \otimes \mathcal{R}$, is integer-valued, i.e. an element of $\mathcal{R}$. On the other hand, we already know the image of $ev_n$ on $\LambdaG$ is contained in $Z(\mathbb{Z}\Gamma \wr S_n)$ and hence integral.

\end{proof}

\begin{proposition}
Recall that $\FHG$ is equipped with two filtrations (the moving and transposition filtrations). If we take the leading order term of $\Psi(m_{\bs\mu})$ with respect to the transposition filtration, and then the leading order term with respect to the  moving filtration, we get $X_{\bm\hat{\bs\mu}}$.
\end{proposition}

\begin{proof}
Consider $ev_n(m_{\bs\mu})$. In the notation of Lemmas \ref{cycle_label_compatibility} and \ref{cycle_label_product},
\[
L_j(c) = \sum_{i<j} X_{(i,j)}(c).
\]
Analogously to the proof of Proposition \ref{murphy_prop}, suppose that $\sigma$ is a cycle containing the index $j$. Multiplying $X_{\sigma}(1)$ by $X_{(i,j)}(c)$ inserts $i$ into the cycle if $i$ is not already in $\sigma$, or removes if it is. The size with respect to the transposition filtration is increased by 1 in the first case, and reduced by 1 in the second case. To find the leading order term of $L_j(1)^r$, we view it as a sum of products of $X_{(i_k,j)}(1)$ ($k=1,\ldots,r$) where each $i_k$ is less than $j$. The maximal degree with respect to the transposition filtration is achieved when the $i_k$ are distinct. In that case, Lemmas \ref{cycle_label_compatibility} and \ref{cycle_label_product} imply that the product of the $X_{(i_k,j)}(1)$ is $X_{\sigma}(1)$, where $\sigma = (i_r, i_{r-1}, \ldots, i_1, j)$ is an $(r+1)$-cycle, and furthermore the leading term of $L_j(1)^r c^{(j)}$ in the transposition filtration is the sum of $X_{\sigma}(c)$ across all $(r+1)$-cycles $\sigma$ whose largest element is $j$. Now, $ev_n(m_{\bs\mu})$ is a sum of products of $L_j(1)^r c^{(j)}$ where the pair of parameters $(r, c)$ occurs $m_r(\bs\mu(c))$ times. Each such term gives us an $(r+1)$-cycle of type $c$.
\newline \newline \noindent
We maximise the degree in the moving filtration when all these cycles are disjoint. In that case, the result is an element of fully-reduced cycle type $\bs\mu$. Every such element arises exactly once, because we may reconstruct the monomial which gave rise to it: an $(r+1)$ cycle whose largest element is $j$ and has type $c$ must have come from $ev_n(x_j^r(c))$.
\end{proof}

\begin{theorem} \label{main_thm}
The map $\Psi$ from Theorem \ref{wreath_map_exists} is an isomorphism and in particular $\FHG = \RGamma \otimes \LambdaG$.
\end{theorem}

\begin{proof}
We see that up to lower order terms, $\Psi$ sends a $\RGamma$-basis of $\RGamma \otimes \LambdaG$, namely $m_{\bs\lambda}$, to a $\RGamma$-basis of $\FHG$, namely $K_{\bm\hat{\bs\lambda}}$. This shows that the homomorphism $\Psi$ is a bijection.
\end{proof}

\section{Applications and Further Directions} \label{applications_section}
\noindent
We prove a wreath-product version of Nakayama's Conjecture. The first step is to understand the action of the wreath-product JM elements on the irreducible (characteristic zero) representations $V^{\bs\lambda}$ of $\Gamma \wr S_n$.
\newline \newline \noindent
Analogously to the GZ basis for irreducible representations of symmetric groups, there is a similar construction for $\Gamma \wr S_n$. It is not a basis, but rather a decomposition into subspaces. What is important for us is that the wreath-product JM elements act by scalars on the GZ subspaces of an irreducible representation of $\Gamma \wr S_n$ by certain scalars.

\begin{definition}[Mishra-Srinivasan, Section 1 \cite{MS}]
Let $\bs\lambda$ be a $\Gamma^*$-indexed multipartition of size $n$. A \emph{Young $\Gamma$-tableau} $T$ of shape $\bs\lambda$ is a labelling of the boxes of all partitions in $\bs\lambda$ with numbers. Suppose that the labelling uses the numbers $\{1, \ldots, n\}$ (with each being used exactly once), such that for each partition on $\bs\lambda$, the labels increase from top to bottom in a column and increase from left to right in a row; in this case we say $T$ is \emph{standard}. If $T$ is a standard Young $\Gamma$-tableau, we write $\chi_{T(i)}$ for the irreducible representation $\chi \in \Gamma^*$ such that $T(\chi)$ contains the box with label $i$, and $c_T(i)$ for the content of the box containing $i$ in the partition corresponding to $\chi_{T(i)}$.
\end{definition}

\begin{example} \label{SYGT_example}
If $\Gamma = C_2$, the following pair of partitions, corresponding to the trivial and sign representations of $C_2$, define a standard Young $C_2$-tableau.
\begin{figure}[H]
\ytableausetup{centertableaux}\begin{ytableau}
1 & 3 & 6 \\
7
\end{ytableau}
\hspace{10mm}
\ytableausetup{centertableaux}\begin{ytableau}
2 & 5 \\
4
\end{ytableau}
\end{figure}
\noindent
In particular, the constituent tableaux of a standard Young $\Gamma$-tableau need not be standard Young tableaux because the labels of the tableau corresponding to $\chi \in \Gamma^*$ are allowed to come from the set $\{1, \ldots, |\bs\lambda|\}$, which may be larger than $\{1, \ldots, |\bs\lambda(\chi)|\}$.
\end{example}

\begin{definition}
For $m \leq n$, let $H_{m,n}$ be the subgroup of $\Gamma \wr S_n$ consisting of elements of the form $(\mathbf{g}, \sigma)$ where $\mathbf{g} \in \Gamma^n$ is arbitrary and $\sigma \in S_m$, where $S_m$ is viewed as the the subgroup of $S_n$ fixing all numbers larger than $m$. Equivalently, $H_{m,n} = \Gamma^n \rtimes S_m$ with the same embedding of $S_m$ into $S_n$.
\end{definition}

\begin{proposition}[Mishra-Srinivasan, Section 4 \cite{MS}]
The restriction of an irreducible representation of $H_{m,n}$ to $H_{m-1,n}$ is multiplicity free (i.e. has no repeated irreducible summands).
\end{proposition}
\noindent
The significance of this multiplicity-free statement is that the decomposition of the restricted representation into irreducibles is unique (since the irreducible summands coincide with isotypic components).

\begin{definition}[Mishra-Srinivasan, Section 3 \cite{MS}]
The \emph{GZ subspaces} of an irreducible representation $V$ of $H_{n,n} = \Gamma \wr S_n$ are defined inductively as follows. An irreducible representation of $H_{m,n}$ splits uniquely as a sum of irreducible representations of $H_{m-1,n}$, each of which may be again restricted. Iterating this for $m=n,n-1, \ldots, 0$, we are left with a canonical decomposition of $V$ into irreducible representations $V_T$ of $H_{0,n} = \Gamma^n$.
\end{definition}

\begin{proposition}[Mishra-Srinivasan, Section 4 \cite{MS}]
The $GZ$ subspaces $V_T$ of the irreducible representation $V^{\bs\lambda}$ are indexed by standard Young $\Gamma$-tableaux $T$ of shape $\bs\lambda$. We have
\[
V^{\bs\lambda} = \bigoplus_T V_T,
\]
where the sum is across standard Young $\Gamma$-tableaux of shape $\bs\lambda$. As a representation of $\Gamma^n$, \[
V_T = \chi_{T(1)} \otimes \cdots \otimes \chi_{T(n)}.
\]
\end{proposition}

\begin{example}
For the standard Young $C_2$-tableau $T$ in Example \ref{SYGT_example}, the corresponding $GZ$ subspace is isomorphic to
\[
\mathbb{C}_+ \otimes \mathbb{C}_- \otimes \mathbb{C}_+ \otimes \mathbb{C}_- \otimes \mathbb{C}_- \otimes \mathbb{C}_+ \otimes \mathbb{C}_+,
\]
where $\mathbb{C}_+$ is the trivial representation and $\mathbb{C}_-$ is the sign representation.
\end{example}

\begin{proposition} \label{wr_JM_eval_prop}
The wreath-product JM element $L_j(c)$ acts the GZ subspace $V_T$ via scalar multiplication by
\[
\frac{|\Gamma|}{\dim(\chi_{T(j)})} c_T(i) \omega_c^{\chi_{T(j)}}.
\]
\end{proposition}
\begin{proof}
Theorem 6.5 of \cite{MS} asserts that $L_j(1)$ acts by the scalar $\frac{|\Gamma|}{\dim(\chi_{T(j)})} c_T(i)$. But $L_j(c) = L_j(1) c^{(j)}$, and $c^{(j)}$ acts on $\chi_{T(1)} \otimes \cdots \otimes \chi_{T(n)}$ as multiplication by $\omega_c^{\chi_{T(j)}}$.
\end{proof}

\begin{theorem} \label{wreath_nakayama_thm}
Suppose that $p$ is a prime number, and suppose that $B_i$ are the $p$-blocks of $\Gamma$ (the subsets of $\Gamma^*$ with a given central character in characteristic $p$). Let $\bs\lambda$ and $\bs \mu$ be two $\Gamma^*$-indexed multipartitions of size $n$. Then the irreducible representations $V^{\bs\lambda}, V^{\bs\mu}$ of $\Gamma \wr S_n$ are in the same $p$-block if and only both of the following conditions hold:
\begin{enumerate}
\item For all $p$-blocks $B_i$ of $\Gamma$,
\[
\sum_{\chi \in B_i} |\bs\lambda(\chi)| = \sum_{\chi \in B_i} |\bs\mu(\chi)|.
\]
\item For each $B_i$ containing only one irreducible $\chi$, the partitions $\bs\lambda(\chi)$ and $\bs\mu(\chi)$ have the same $p$-core.
\end{enumerate}
\end{theorem}
\begin{proof}
Analogously to the proof of Nakayama's conjecture (Theorem \ref{nakayama_conjecture}), we fix a $p$-modular system $(K, \mathcal{O}, k)$, and let $\pi$ be a uniformiser of $\mathcal{O}$. We determine when two central characters are equal modulo $(\pi)$, where we use the surjection $ev_n: \RGamma \otimes \LambdaG \to Z(\mathbb{Z}\Gamma \wr S_n)$ to compute the action of the centre. The action of $\RGamma \otimes \LambdaG$ is determined by the individual actions of $\RGamma$ and $\LambdaG$. 
\newline \newline \noindent
First we consider the action of $\RGamma$, whose image in $Z(\mathbb{Z}\Gamma \wr S_n)$ is contained in $Z(\mathbb{Z} \Gamma^n)$ (see Theorem \ref{r_gamma_hom_thm}). So to compute the action by scalars, we may restrict the representation $V^{\bs\lambda}$ from $\Gamma \wr S_n$ to $\Gamma^n$. This gives the sum of the GZ subspaces $V_T$, each of which is of the form
\[
\chi_{T(1)} \otimes \cdots \otimes \chi_{T(n)},
\]
where the number times a given irreducible $\chi \in \Gamma^*$ appears is $|\bs\lambda(\chi)|$. We compute the action of $\RGamma$ on a GZ subspace $V_T$. By Lemma \ref{specialised_gen_fun_lemma},
\[
\sum_{\mathbf{N} } ev_n(B_{\mathbf{N}}) x^{\mathbf{N}} = (1 + \sum_c c x_c)^{\otimes n}.
\]
Now, the action of $(1 + \sum_c c x_c)$ on $\chi$ is by multiplication by
\[
(1 + \sum_c \omega_{c}^\chi x_c),
\]
where $\omega_c^\chi$ is the central character of $\Gamma$ corresponding to the action of $c$ on $\chi$. Hence the action of $\sum_{\mathbf{N} } ev_n(B_{\mathbf{N}}) x^{\mathbf{N}}$ is
\[
\prod_{\chi \in \Gamma^*} (1 + \sum_c \omega_{c}^\chi x_c)^{|\bs\lambda(\chi)|}.
\]
This result remains true after modular reduction (i.e. passing to $k = \mathcal{O}/(\pi)$), although this has the effect of making central characters within a single $p$-block coincide. Since the $B_{\mathbf{N}}$ are a basis of $\RGamma$, the action of $\RGamma$ is determined by this generating function which is actually a polynomial. Polynomials over a field form a unique factorisation domain, so two such polynomials agree if and only if they have the same factors (note that all our factors are normalised to have constant term 1, so two factors are associates if and only if they are equal). But two factors for $\chi_1$ and $\chi_2$, coincide precisely when
\[
(1 + \sum_c \omega_{c}^{\chi_1} x_c) = (1 + \sum_c \omega_{c}^{\chi_2} x_c),
\]
i.e. when $\omega_c^{\chi_1} = \omega_c^{\chi_2}$, which is to say that the central characters of $\chi_1$ and $\chi_2$ agree. This in turn happens when $\chi_1$ and $\chi_2$ are in the same $p$-block. By factoring the generating function we determine the multiplicity of the $\chi$ from each $p$-block of $\Gamma$. We conclude that the $\RGamma$ acts by the same scalars (in $k$) on $V^{\bs\lambda}$ and $V^{\bs\mu}$ if and only if condition $(1)$ holds.
\newline \newline \noindent
To understand the action of $\LambdaG$, we again consider a GZ subspace $V_T$. By Proposition \ref{wr_JM_eval_prop}, the wreath-product JM element $L_j(c)$ acts on $V_T$ by the scalar
\[
\frac{|\Gamma|}{\dim(\chi_{T(j)})} c_T(j) \omega_c^{\chi_{T(j)}}.
\]
If $\frac{|\Gamma|}{\dim(\chi_{T(j)})}$ is a multiple of $p$, then this scalar is zero in $k$. Now suppose that conditions (1) and (2) are both satisfied for $\bs \lambda$ and $\bs \mu$. For each $\chi \in \Gamma_*$ and $c \in \Gamma_*$, the multiset of scalars
\[
\frac{|\Gamma|}{\dim(\chi)} c_T(j) \omega_c^{\chi}
\]
for $j$ such that $\chi_{T(j)} = \chi$ coincides for $T$ of shape $\bs \lambda$ and $\bs\mu$. To check this, we consider whether $\frac{|\Gamma|}{\dim(\chi)}$ is a multiple of $p$. If it is a multiple of $p$, then all these scalars are zero. If $\frac{|\Gamma|}{\dim(\chi)}$ is not a multiple of $p$, then the contents $c_T(j)$ coincide because $\bs\lambda(\chi)$ and $\bs\mu(\chi)$ have the same $p$-core. As a result, any element of $\LambdaG$ acts by the same scalar modulo $(\pi)$ on $V^{\bs\lambda}$ and $V^{\bs\mu}$. We conclude that the ``if'' implication of the theorem holds.
\newline \newline \noindent
For $\chi$ such that $|G|/\dim(\chi)$ is not divisible by $p$, the idempotent
\[
e_\chi = \frac{\dim(\chi)}{|G|}\sum_{g \in \Gamma} \chi(g)g^{-1}
\]
is well-defined in $k$. We may therefore consider the elements $L_j(e_\chi)$ in $k\Gamma \wr S_n$. By Proposition \ref{wr_JM_eval_prop}, the action of $L_j(e_\chi)$ on $V_T$ is
\[
\frac{|\Gamma|}{\dim(\chi)} c_T(j)
\]
if $\chi_{T(j)} = \chi$, and zero otherwise. This is because $\chi$ is in a $p$-block by itself (see Proposition \ref{brauer_nesbitt_prop}), and $e_\chi$ acts by zero on irreducibles in other $p$-blocks, i.e. on all other irreducibles.
\newline \newline \noindent
Just as in the proof of Theorem \ref{wreath_map_exists}, the elementary symmetric functions $e_r$ evaluated in $L_j(e_\chi)$,
\[
e_r(L_1(e_\chi), \ldots, L_n(e_\chi)),
\]
are well defined elements of $k \otimes \LambdaG$. So the action of $\LambdaG$ on $V_T$ determines the action of the generating function
\[
\sum_{r \geq 0} t^r e_r(L_1(e_\chi), \ldots, L_n(e_\chi)) = \prod_{j=1}^n (1 + tL_j(e_\chi)).
\]
But we have determined the action of $L_j(e_\chi)$, so we know that this generating function acts by
\[
\prod_{j} (1 + t \frac{|\Gamma|}{\dim(\chi)} c_T(j))
\]
where the product is only over those $j$ such that $\chi_{T(j)} = \chi$. Appealing to the unique factorisation of polynomials over $k$, the action of $\LambdaG$ determines the multiset of scalars $\frac{|\Gamma|}{\dim(\chi)} c_T(j)$ (viewed as elements of $k$). Now, $\frac{|\Gamma|}{\dim(\chi)}$ it is invertible in $k$, so we have determined the multiset $c_T(j)$ for $j$ such that $\chi_{T(j)} = \chi$, i.e. the contents of $\bs\lambda(\chi)$, viewed as elements of $k$. Together with $|\bs\lambda(\chi)|$, which was determined by the action of $\RGamma$, this determines the $p$-core of $\bs\lambda(\chi)$, which completes the proof of the ``only if'' direction.
\end{proof}
\begin{remark}
Appendix B to Chapter 1 of \cite{Macdonald} computes the characters of $\Gamma \wr S_n$ in terms of the characters of $\Gamma$ and $S_n$. This can be used to give an expression for the central characters, and in turn to prove Theorem \ref{wreath_nakayama_thm}.
\end{remark}

\begin{example}
As pointed out on the first page of \cite{DipperJames}, over a field of characteristic 2, every irreducible representation of $C_2 \wr S_n$ can be obtained by pulling back an irreducible representation of $S_n$ via the quotient map $C_2 \wr S_n \to S_n$. This does not, however, mean that the representation categories of $C_2 \wr S_n$ and $S_n$ are equivalent.  Note that there is only one 2-block $B_i$ of $C_2$, and this block is not semisimple. So in Theorem \ref{wreath_nakayama_thm}, condition (2) is vacuous, and condition (1) holds automatically since $\sum_{\chi \in B_i} |\bs\lambda(\chi)| = n$ as there is only one block. As a result $C_2 \wr S_n$ has only one 2-block for any $n$. By comparison, Example \ref{2-core_example} shows $S_3$ has two 2-blocks. In general, this argument shows that when $\Gamma$ is a non-trivial $p$-group, $\Gamma \wr S_n$ will have only one $p$-block.
\end{example}

\begin{remark}
It is straightforward to generalise the character symmetric function from Definition \ref{char_sym_fun_def} to the case of wreath products. For a partially-reduced conjugacy class $\bs\mu$, we have an element of $\RGamma \otimes \LambdaG$ given by $\psi^{-1}(K_{\bs\mu})$. Its ``content evaluation'' as in the proof of Theorem \ref{wreath_nakayama_thm} will yield central characters of wreath products $\Gamma \wr S_n$.
\end{remark}

\appendix
\section{Properties of \texorpdfstring{$\RGamma$}{R-Gamma} and \texorpdfstring{$\LambdaG$}{Lambda-Gamma}}
\noindent
In this appendix we show that both $\RGamma$ and $\LambdaG$ are Hopf algebras. We also explain that $\RGamma$ is the algebra of distributions on $\mathbb{Z}\Gamma_*$ (viewed as an affine monoid scheme). Furthering the analogy with integer-valued polynomials, we classify homomorphisms from $\RGamma$ to a field; in positive characteristic, they are indexed by several $p$-adic parameters, one for each $p$-block of $\Gamma$.
\newline \newline \noindent
First we show that $\LambdaG$ is a Hopf algebra.
\begin{theorem}
There is a Hopf algebra structure on $\LambdaG$ with comultiplication
\[
\Delta(m_{\bs\lambda}) = \sum_{\bs\mu, \bs\nu} m_{\bs\mu} \otimes m_{\bs\nu},
\]
where the sum is over all pairs of multipartitions $\bs\mu, \bs\nu$ such that $m_i(\bs\mu(c)) + m_i(\bs\nu(c)) = m_i(\bs\lambda(c))$ for all $i \in \mathbb{Z}_{>0}$ and $c \in \Gamma_*$. The counit $\varepsilon$ obeys $\varepsilon(m_{\bs\mu}) = 0$ when $|\bs\mu| > 0$ (and $\varepsilon(1) = 1$).
\end{theorem}
\begin{proof}
We have that $(\Delta \otimes 1) \circ \Delta (m_{\bs\lambda}) = (1 \otimes \Delta) \circ \Delta (m_{\bs\lambda})$, both being equal to
\[
\sum_{\bs\mu, \bs\nu,\bs\rho} m_{\bs\mu} \otimes m_{\bs\nu} \otimes m_{\bs\rho},
\]
where the sum is over all multipartitions obeying $m_i(\bs\mu(c)) + m_i(\bs\nu(c)) + m_i(\bs\rho(c)) = m_i(\bs\lambda(c))$. This verifies that $\Delta$ is coassociative. To see that the $\varepsilon$ defines the counit of a bialgebra, we note that 
\[
(\varepsilon \otimes 1) \circ \Delta (m_{\bs\lambda}) = \sum_{\bs\mu, \bs\nu} \varepsilon(m_{\bs\mu}) \otimes m_{\bs\nu},
\]
and $\varepsilon(m_{\bs\mu})$ is nonzero only if $m_i(\bs\mu(c)) = 0$ for all $i$ and $c$, in which case $m_{\bs\nu} = m_{\bs\lambda}$. The case of $(1 \otimes \varepsilon) \circ \Delta$ is identical. Finally, we must show there is an antipode $S$. We do not give an explicit formula, but we prove it exists by induction on $|\bs\lambda|$. We take $S(1) = 1$, which serves as the base case $|\bs\lambda| = 0$. For $|\bs\lambda| \geq 1$, we write
\[
\Delta(m_{\bs\lambda}) = m_{\bs\lambda} \otimes 1 + \sum_{\bs\mu,\bs\nu} m_{\bs\mu} \otimes m_{\bs\nu},
\]
where we have written the term where to $\bs\nu$ is the empty partition (and $\bs\mu = \bs\lambda$) separately. The antipode axiom becomes
\[
S(m_{\bs\lambda}) + \sum_{\bs\mu, \bs\nu} S(m_{\bs\mu}) m_{\bs\nu} = 0.
\]
But every term in the sum has $|\bs\mu| < |\bs\lambda|$, so this serves to define $S(m_{\bs\lambda})$ inductively. By construction, the resulting map $S$ obeys the antipode axiom.
\end{proof}

\begin{remark}
The Hopf algebra structure on $\LambdaG$ agrees with the usual Hopf algebra structure on $\Lambda$ (see Section 1.5, Example 25 of \cite{Macdonald}) when we set $\Gamma$ to be the trivial group. In fact, it can be shown that the isomorphism
\[
\mathbb{C} \otimes \LambdaG = \mathbb{C} \otimes \Lambda^{\otimes |\Gamma^*|}
\]
from Proposition \ref{complex_bas_change_prop} is an isomorphism of Hopf algebras,  where $\Lambda^{\otimes |\Gamma^*|}$ is a tensor product of Hopf algebras, and is therefore a Hopf algebra.
\end{remark}

\begin{remark}
One of the most important tools in the theory of symmetric functions is the Cauchy identity:
\[
\prod_{i,j} \frac{1}{1-x_iy_j} = \sum_{\lambda} s_\lambda(x) s_\lambda(y).
\]
It would be desirable to have a similar identity for $\LambdaG$. However, the most natural formulation of the Cauchy identity is as a formula for the series
\[
\sum_i b_i \otimes b_i^*,
\]
where $b_i$ is a basis of $\Lambda$ and $b_i$ is the dual basis with respect to the Hall inner product on $\Lambda$. The equation above takes both $b_i$ and $b_i^*$ to be the Schur functions $s_\lambda$, which are self dual. For this notion to make sense in $\LambdaG$, we need an inner product on $\LambdaG$. The natural choice, $\langle -,- \rangle$, comes from the isomorphism
\[
\mathbb{C} \otimes \LambdaG = \mathbb{C} \otimes \Lambda^{\otimes |\Gamma^*|},
\]
where each tensor factor of $\Lambda$ is given the Hall inner product. The reason for this choice is that the Hall inner product defines a Hopf pairing on $\Lambda$, and since the Hopf algebra structure on $\LambdaG$ comes from the tensor product of Hopf algebra strutures on $\Lambda$, the tensor product of Hall inner products defines a Hopf pairing on $\mathbb{C} \otimes \LambdaG$. Now we encounter a problem: $\LambdaG$ is not self dual with respect to this inner product: if $b_i$ is a basis of $\LambdaG$, then $b_i^*$ will be contained in
\[
\LambdaG^* = \{f \in \mathbb{C} \otimes \Lambda^{\Gamma_*} \mid \langle f, g \rangle \in \mathbb{Z} \mbox{ for all $g \in \LambdaG$}\},
\]
which is not a subset of $\LambdaG$; roughly speaking it is the ``dual'' of $\LambdaG$, so $\LambdaG^*$ inherits the structure of a Hopf algebra. It would be interesting to give an independent description of $\LambdaG^*$.
\end{remark}
\noindent
It would also be interesting to have a presentation of $\LambdaG$ by generators and relations. The argument from Corollary 2.3 of \cite{Vaccarino} (filtration/induction on $l(\bs\lambda)$) shows that $\LambdaG$ is generated by $m_{(r^n)_c}$ for $r,n \in \mathbb{Z}_{>0}$ and $c \in \Gamma_*$. However the argument from Proposition 2.5 of \cite{Vaccarino} can be adapted to show that there are redundancies; the generator $m_{(r^n)_c}$ may be omitted if $c$ is an $n$-th power in the ring $\mathbb{Z}\Gamma_*$.
\newline \newline \noindent
Now we turn our attention to $\RGamma$. Let $R$ be an arbitrary commutative ring, and let us view $R\Gamma_* = \hom_{alg}(A, R)$ as the affine scheme represented by $A = \mathbb{Z}[t_{c_1}, \ldots, t_{c_l}]$, where the polynomial variables are indexed by elements of $\Gamma_*$. Now, $R\Gamma_*$ is a ring whose addition and multiplication maps are algebraic, i.e. the corresponding maps
\[
R\Gamma_* \times R\Gamma_* \to R\Gamma_*
\]
are morphisms of affine schemes. This amounts to saying that $R\Gamma_*$ is a (commutative) ring object in the category of affine schemes. On the level of coordinate rings, we get homomorphisms $\Delta^{(+)}$, $\Delta^{(\times)}$ from $A$ to $A \otimes A$. Explicitly, they are defined by
\[
\Delta^{(+)}(t_c) = t_c \otimes 1 + 1 \otimes t_c
\]
and
\[
\Delta^{(\times)}(t_c) = \sum_{d,e \in \Gamma_*} A_{d,e}^c t_d \otimes t_e.
\]
We call these two maps the \emph{coaddition} and \emph{comultiplication}, respectively. They satisfy relations coming from the associative and distributive laws in a ring. Because of these two maps, a commutative ring object in the category of affine schemes is sometimes called a \emph{biring}, although this is not the same thing as a bialgebra.
\newline \newline \noindent
Now we let $I$ be the ideal of $A$ vanishing at the multiplicative identity; $I$ is generated by $t_1-1$ and $t_c$ for $c \neq 1$. We have $\Delta^{(\times)}(I) \subseteq I \otimes A + A \otimes I$ (this follows from the fact that $I$ is the ideal corresponding to the identity element of $R\Gamma_*$, but can also be checked directly). We use $I$ to define the distribution algebra of $R\Gamma_*$ (see Chapter 7 of Part I of \cite{Jantzen_alggp}).
\begin{definition}
The \emph{distribution algebra} of $R\Gamma_*$ is
\[
D(R\Gamma_*) = \bigcup_{n \geq 0} (A/I^n)^*,
\]
where $(A/I^n)^*$ is viewed as a subspace of $A^* = \hom_{\mathbb{Z}}(A, \mathbb{Z})$.
\end{definition}
\noindent
So $D(R\Gamma_*)$ is the set of linear functionals on $A$ that vanish on some power of $I$. We obtain a multiplication on $D(R\Gamma_*)$ by dualising $\Delta^{(\times)}$. If $f_1, f_2 \in D(R\Gamma_*)$, then we take
\[
(f_1f_2)(a) = (f_1 \otimes f_2)(\Delta^{(\times)}(a))
\]
for $a \in A$. The product $f_1f_2$ vanishes on some power of $I$ because we have
\[
\Delta^{(\times)}(I^n) \subseteq \sum_{m=0}^n I^m \otimes I^{n-m},
\]
and taking $n$ sufficiently large we guarantee that each summand is annihilated by $f_1 \otimes f_2$. Finally, we see that the coassociativity of $\Delta^{(\times)}$ implies the associativity of the multiplication, and the canonical map $A \to A/I = \mathbb{Z}$ is the multiplicative identity. Thus $D(R\Gamma_*)$ is an associative ring.
\begin{theorem} \label{dist_thm}
We have that $D(R\Gamma_*) = \RGamma$. 
\end{theorem}
\begin{proof}
Define the variables $z_1 = t_1 - 1$ and $z_c = t_c$ for $c \neq 1$. The coordinate ring $A$ has a basis consisting of monomials $\prod_{c \in \Gamma_*} z_c^{\mathbf{N}(c)}$, indexed by $\mathbf{N} \in \mathbb{Z}_{\geq 0}^{\Gamma_*}$. Such a monomial is contained in $I^{|\mathbf{N}|}$. Hence $D(R\Gamma_*)$ has $\mathbb{Z}$-basis consisting of the dual basis, which we denote $\beta_{\mathbf{N}}$. The structure constants of the multiplication with respect to the basis $\beta_{\mathbf{N}}$ are precisely the structure constants of $\Delta^{(\times)}$ with respect to the monomials $\prod_c z_c^{\mathbf{N}_c}$. To compute the multiplication, note that for any $c \in \Gamma_*$,
\[
\Delta^{(\times)} (z_c) = z_c \otimes 1 + 1 \otimes z_c + \sum_{d,e}A_{d,e}^c z_d \otimes z_e.
\]
We recognise this as the change of variables appearing in Proposition \ref{r_gamma_is_an_algebra_prop} (subject to the relabelling $z_i \otimes 1 = x_i$ and $1 \otimes z_i = y_i$) that describes the structure constants of the multiplication in $\RGamma$. This implies that the linear map $\RGamma \to D(R\Gamma_*)$ given by $B_{\mathbf{N}} \mapsto \beta_{\mathbf{N}}$ is a ring isomorphism.
\end{proof}
\noindent
It is well known that over a field of characteristic zero, the distribution algebra of an algebraic group $G$ coincides with the universal enveloping algebra of the Lie algebra of $G$. Proposition \ref{r_gamma_is_an_algebra_prop} witnesses this fact; the proof explains that $\mathbb{Q} \otimes \RGamma = \mathcal{U}(\mathbb{Q}\Gamma_*)$. Although $R\Gamma_*$ is a monoid rather than a group, we may consider the group of units, which has same distribution algebra. Since distribution algebras of algebraic groups are Hopf algebras, the same is true of $\RGamma$.
\begin{proposition}
We have that $\RGamma \subseteq \mathcal{U}(\mathbb{Q}\Gamma_*)$ is a Hopf subalgebra.
\end{proposition}
\noindent
Now we turn our attention to classifying maps into a field. For context and background, we direct the reader to \cite{HarmanHopkins}. We write $\mathbb{Z}_p$ for the set of $p$-adic integers. The motivating example is the following.
\begin{example}[Section 8 \cite{HarmanHopkins}] \label{hom_classification_example}
Homomorphisms from the ring $\mathcal{R}$ of integer-valued polynomials to a field $\mathbb{F}$ are given by the following. If $\mathbb{F}$ has characteristic zero, they are given by evaluating the polynomial variable $t$ at an arbitrary element of $\mathbb{F}$. If $\mathbb{F}$ has characteristic $p>0$, then they are given by evaluating the polynomial variable $t$ at a $p$-adic integer. Evaluation of an integer-valued polynomial at a $p$-adic integer yields another $p$-adic integer, which may be viewed as an element of $\mathbb{F}$ via the canonical map
\[
\mathbb{Z}_p \to \mathbb{Z}_p/(p) = \mathbb{Z}/p\mathbb{Z} \subseteq \mathbb{F}.
\]
In fact, if $t \in \mathbb{Z}_p$, ${t \choose p^r}$ is the $r$-th $p$-adic digit of $t$ (viewed as an element of $\mathbb{F}$).
\end{example}
\noindent
First we get the characteristic-zero case out of the way.
\begin{proposition}
Suppose that $\mathbb{F}$ is a field of characteristic zero. Homomorphisms $\RGamma \to \mathbb{F}$ are given by evaluating each $T(c)$ at an arbitrary element of $\mathbb{F}$ (and are hence parametrised by $\mathbb{F}^{\Gamma_*} = \mathbb{F}\Gamma_*$).
\end{proposition}
\begin{proof}
We have $\mathbb{Q} \subseteq \mathbb{F}$, and $\mathbb{Q} \otimes \RGamma = \mathcal{U}(\mathbb{Q}\Gamma_*)$ is the free polynomial algebra in variables $T(c)$ for $c \in \Gamma_*$. Homomorphisms from a polynomial algebra to $\mathbb{F}$ amount to evaluating each variable at an element of $\mathbb{F}$.
\end{proof}
\noindent
Now we move on to the positive characteristic case. Fix a prime $p$, and let $\mathbb{F}$ be a field of characteristic $p$. Since any field embeds in its algebraic closure, we may as well take $\mathbb{F}$ to be algebraically closed. 
\begin{definition}
Let $q \in \mathbb{F}\Gamma_*$ be $q = \sum_{c \in \Gamma_*} m_c c$. We define
\[
B_{q,r} = \sum_{|\mathbf{M}| = p^r} \prod_{c \in \Gamma_*} m_c^{\mathbf{M}(c)} B_{\mathbf{M}}.
\]
\end{definition}
\noindent
For example, if $q=c \in \Gamma_*$, we have
\[
B_{c,r} = B_{\mathbf{M}}
\]
where $\mathbf{M}(c) = p^r$ and $\mathbf{M}(c^\prime) = 0$ for $c^\prime \neq c$. Also, $B_{0,r} = 0$ for all $r$.

\begin{lemma} \label{rgamma_modular_generation_lemma}
The algebra $\mathbb{F} \otimes \RGamma$ is generated by the elements $B_{q,r}$, where $q$ varies across a basis of $\mathbb{F}\Gamma_*$ and $r \in \mathbb{Z}_{\geq 0}$,
\end{lemma}
\begin{proof}
Lemma \ref{leading_term_lemma} asserts that $B_{\mathbf{M}}$ has leading order term
\[
\prod_{c \in \Gamma_*} \frac{T(c)^{\mathbf{M}(c)}}{\mathbf{M}(c)!}.
\]
Writing $q = \sum_{c \in \Gamma_*} m_c c$ and using the multinomial theorem, we see that the leading order term of $B_{q,r}$ is
\[
\sum_{|\mathbf{M}| = p^r} \prod_{c \in \Gamma_*} m_c^{\mathbf{M}(c)} \frac{T(c)^{\mathbf{M}(c)}}{\mathbf{M}(c)!} = \frac{T(q)^{p^r}}{p^r!}.
\]
If the $p$-adic decomposition of $n$ is $\sum_r n_r p^r$, then over $\mathbb{F}$ we have
\[
\prod_r \left(\frac{T(q)^{p^r}}{p^r!}\right)^{n_r} = \frac{n!}{\prod_{r} (p^r!)^{n_r}} \cdot \frac{T(q)^n}{n!}.
\]
Here, the multinomial coefficient
\[
\frac{n!}{\prod_{r} (p^r!)^{n_r}}
\]
is invertible in $\mathbb{F}$. This implies that the leading order term of any $B_{\mathbf{M}}$ is a (scalar multiple of a) product of leading order terms of the $B_{c,r}$, where we are using the particular basis of conjugacy class sums, proving the lemma for that particular basis. To prove the lemma for an arbitrary basis $q_i$, we express $B_{c,r}$ in terms of the $B_{q_i,r}$. Let $c = \sum_i k_i q_i$, and observe that to leading order, $B_{c,r}$ may be written
\[
\frac{T(\sum_i k_i q_i)^{p^r}}{p^r!} = \sum_{|\mathbf{W}| = p^r} \prod_{i} k_i^{\mathbf{W}(q_i)} \frac{T(q_i)^{\mathbf{W}(q_i)}}{\mathbf{W}(q_i)!},
\]
where the sum is over functions $\mathbf{W}$ from the basis $q_i$ to $\mathbb{Z}_{\geq 0}$, such that $\sum_i \mathbf{W}(q_i) = p^r$. We observe that this is equal to
\[
\sum_i k_i^{p^r} \frac{T(q_i)^{p^r}}{p^r!}
\]
plus ``mixed'' terms whose exponents (values of $\mathbf{W}(q)$) are less than $p^r$ and may therefore be written in terms of $B_{q,s}$ for $s <r$. By induction on $r$, the generating set $B_{c,r}$ is contained in the subring generated by all elements of the form $B_{q,r}$, and this completes the proof.
\end{proof}

\begin{lemma}
View the map $ev_n^{\mathcal{R}}$ from Theorem \ref{r_gamma_hom_thm} as having domain $\mathbb{F} \otimes \RGamma$ and codomain $(\mathbb{F}\Gamma_*)^{\otimes n}$. We have that $ev_n^{\mathcal{R}}(B_{q,r})$ is equal to the sum of all pure tensors with $p^r$ entries set to $q$, and the remainder set to $1$.
\end{lemma}
\begin{proof}
Expanding in terms of the basis $\Gamma_*$,
\[
q^{\otimes p^r} = (\sum_c m_c c)^{\otimes p^r},
\]
we get each tensor monomial with $\mathbf{M}(c)$ factors equal to $c$ with multiplicity 
in the conjugacy class sums $c$ with multiplicity $\prod_{c \in \Gamma_*} m_c^{\mathbf{M}(c)}$. Inserting $n-p^r$ tensor factors of $1$ in all possible ways completes the proof.
\end{proof}

\begin{lemma}
For any $q \in \mathbb{F}\Gamma_*$, we have the following equation in $\mathbb{F} \otimes \RGamma$;
\[
B_{q,r}^p =  B_{q^p,r},
\]
where $q^p$ is viewed as an element of $\mathbb{F}\Gamma_*$.
\end{lemma}
\begin{proof}
We use Theorem \ref{r_gamma_hom_thm} and the fact that the $b_{\mathbf{N}}$ are linearly independent. So it is enough to show that 
\[
ev_n^{\mathcal{R}}(B_{q,r})^p = ev_n^{\mathcal{R}}(B_{q^p,r})
\]
over $\mathbb{F}$ for all $n$. We note that $ev_n^{\mathcal{R}}(B_{q,p^r})$ is the sum of all pure tensors (with respect to the basis $\Gamma_*$) with $p^r$ factors set to $q$, and the remainder set to $1$ (see Lemma \ref{specialised_gen_fun_lemma}). Raising this sum to the $p$-th power is the same as raising each pure tensor to the $p$-th power because the binomial coefficients from the cross-terms vanish in characteristic $p$. So we have the sum of pure tensors with $p^r$ factors set to $q^p$, which agrees with $ev_n^{\mathcal{R}}(B_{q^p,r})$.
\end{proof}

\begin{proposition} \label{modular_evaluation_prop}
Let $\varphi: \mathbb{F} \otimes \RGamma \to \mathbb{F}$ be a ring homomoprhism. If $q \in J(\mathbb{F}\Gamma_*)$, then $\varphi(B_{q,r}) = 0$, where $J$ is the Jacobson radical. If $e \in \mathbb{F}\Gamma_*$ is idempotent, then $\varphi(B_{e,r})$ is an element of the prime subfield of $\mathbb{F}$ (isomorphic to $\mathbb{Z}/p\mathbb{Z}$).
\end{proposition}

\begin{proof}
Because $\mathbb{F}\Gamma_*$ is a finite-dimensional algebra, its Jacobson radical is nilpotent. This means that $J(\mathbb{F}\Gamma_*)^{p^k} = 0$ for some $k$. Hence for $q$ in the Jacobson radical,
\begin{eqnarray*}
\varphi(B_{q,r})^{p^k} 
&=& \varphi(B_{q,r}^{p^k}) \\ 
&=& \varphi(B_{q^{p^k},r}) \\ 
&=& \varphi(B_{0,r}) \\
&=& 0.
\end{eqnarray*}
Since the only nilpotent element of a field is zero, $\varphi(B_{q,r}) = 0$. If $e$ is idempotent,
\[
\varphi(B_{e,r})^p = \varphi(B_{e^p,r}) = \varphi(B_{e,r}).
\]
This means $\varphi(B_{q,r})$ solves the equation $x^p = x$ (i.e. is Frobenius-fixed), but such solutions are precisely elements of the prime subfield of $\mathbb{F}$.
\end{proof}
\noindent
To avoid conflicts of notation, instead of writing $B_i$ for the blocks of $\mathbb{F}\Gamma$, we instead write $C_u$. Correspondingly, we write $\omega_c^{u} \in \mathbb{F}$ for the central character of the block $C_u$ evaluated at the conjugacy class sum $c$.

\begin{theorem}
Homomorphisms $\varphi: \RGamma \to \mathbb{F}$ are given by evaluating 
\[
T(c) \mapsto \sum_{u} \omega_c^{u} t_u,
\]
where the sum is over the blocks $C_u$ of $\mathbb{F}\Gamma$, and $t_u \in \mathbb{Z}_p$ are arbitrary $p$-adic integers (one for each block $C_u$).
\end{theorem}
\begin{proof}
For each block idempotent $e_u \in C_u$ ($e_u$ is the identity element of $C_u$, viewed as an element of $\mathbb{F}\Gamma_*$), $\varphi(B_{e_u, r})$ is an element of $\mathbb{Z}/p\mathbb{Z}$. Thus
\[
t_u = \sum_r \varphi(B_{e_u, r}) p^r
\]
are well-defined $p$-adic integers. First of all, we show that the $t_u$ determine $\varphi$. It is enough to show that the $t_u$ determine $\varphi$ on the generating set from Lemma \ref{rgamma_modular_generation_lemma}. As discussed in the proof of Proposition  \ref{central_char_blocks}, $Z(C_u)/J(Z(C_u)) = \mathbb{F}$, and since $J(\mathbb{F}\Gamma_*) = \bigoplus_u J(Z(C_u))$, any basis of $J(\mathbb{F}\Gamma_*)$ together with the identity elements $e_u$ of the blocks $C_u$ form a basis of $\mathbb{F}\Gamma_*$. Let $q$ be an element of this basis of $\mathbb{F}\Gamma_*$. Lemma \ref{rgamma_modular_generation_lemma} implies that $B_{q,r}$ generates $\mathbb{F} \otimes \RGamma$, but Proposition \ref{modular_evaluation_prop} implies that $\varphi(B_{q,r})=0$ for $q$ in $J(\mathbb{F}\Gamma_*)$. So $\varphi$ is determined by the values on $B_{e_u,r}$, and this is precisely what is encoded in the $t_u$.
\newline \newline \noindent
It remains to show that any choice of the $t_u$ gives rise to a homomorphism $\varphi$. We describe it via an equality of generating functions:
\[
\sum_{\mathbf{N}} \varphi(B_{\mathbf{N}}) x^{\mathbf{N}} = \prod_u (1 + \sum_c \omega_c^{u} x_c)^{t_u}.
\]
Here we are using the series expansion
\[
(1 + x)^t = \sum_r x^r {t \choose r}.
\]
Because an integer-valued polynomial (e.g. ${t \choose r}$) evaluated at a $p$-adic integer is again a $p$-adic integer, $\varphi(B_{\mathbf{N}})$ is a well-defined element of $\mathbb{F}$. It remains to check that $\varphi$ defines a homomorphism, and that it actually corresponds to the parameters $t_u$. To check $\varphi$ is a homomorphism, we first observe that because central characters are themselves homomorphisms $\mathbb{F}\Gamma_* \to \mathbb{F}$,
\[
\sum_{i} \omega_{c_i}^{u} x_{c_i}  \sum_{j} \omega_{c_j}^{u} y_{c_j} = 
\sum_k \sum_{i,j} A_{i,j}^k \omega_{c_k}^{u} x_{c_i}y_{c_j}.
\]
Now we compute
\begin{eqnarray*}
\sum_{\mathbf{N}} \varphi(B_{\mathbf{N}}) x^{\mathbf{N}} \sum_{\mathbf{M}} \varphi(B_{\mathbf{M}}) y^{\mathbf{M}} 
&=& \prod_i (1 + \sum_i \omega_{c_i}^{u} x_{c_i})^{t_u}\prod_i (1 + \sum_j \omega_{c_j}^{u} y_{c_j})^{t_u} \\
&=& \prod_u (1 + \sum_i \omega_{c_i}^{u} x_{c_i} + \sum_j \omega_{c_j}^{u} x_{c_j} + \sum_k \sum_{i,j} A_{i,j}^k \omega_{c_k}^{u} x_{c_i} y_{c_j})^{t_u} \\
&=& \prod_u (1 + \sum_i \omega_{c_i}^{u} (x_{c_i}+ y_{c_i} + \sum_{j,k} A_{j,k}^i x_jy_k))^{t_u} \\
&=& \prod_u (1 + \sum_i \omega_{c_i}^{u} z_{c_i})^{t_u} \\
&=& \sum_{\mathbf{K}} \varphi(B_{\mathbf{K}}) z^{\mathbf{K}},
\end{eqnarray*}
where $z_{c_i} = x_{c_i}+ y_{c_i} + \sum_{j,k} A_{j,k}^i x_jy_k$, as in Section \ref{rgamma_section}. Proposition \ref{theta_product_prop} now implies that $\varphi$ is a homomorphism. It remains to verify that $\varphi(B_{e_u, r})$ is equal to the $r$-th $p$-adic digit of $t_u$. We fix a block $C_v$ and write $e_v = \sum m_c c$. Now, to compute
\[
ev(B_{e_v,r}) = \sum_{|\mathbf{M}| = p^r} \prod_{c \in \Gamma_*} m_c^{\mathbf{M}(c)} ev(B_{\mathbf{M}}),
\]
we take the generating function we used to define $\varphi$ and evaluate at $x_c = \varepsilon m_c$, where $\epsilon$ is a formal variable. This gives
\[
\sum_{\mathbf{N}} \varphi(B_{\mathbf{N}}) \prod_c m_c^{\mathbf{N}(c)} \epsilon^{|\mathbf{N}|} = \prod_u (1 + \sum_c \omega_c^{u} m_c \epsilon)^{t_u}.
\]
We observe that $\sum_{c} \omega_c^{u} m_c$ is the central character of $C_u$ evaluated at $e_v$, and therefore this sum is equal to one if $u = v$, and zero otherwise. This shows that $\varphi(B_{e_v, r})$ is equal to the coefficient of $\epsilon^{p^r}$ in
\[
(1 + \varepsilon)^{t_v}.
\]
But this is precisely the $r$-th $p$-adic digit of $t_v$ by Example \ref{hom_classification_example}.
\end{proof}


\bibliographystyle{alpha}
\bibliography{ref.bib}

\end{document}